\documentclass[11pt, a4paper]{amsart}

\usepackage[
a4paper,
vmargin=2.5cm,
hmargin=3cm
]{geometry}

\usepackage[table]{xcolor}

\usepackage{floatrow}
\usepackage{enumitem} 
\usepackage{caption}
\usepackage{float}
\usepackage{vmargin} 
\usepackage{mathrsfs}
\usepackage{emptypage}
\usepackage{marginnote}
\newcommand{\mc}{\mathcal}

\numberwithin{equation}{section}

\usepackage{calc}

\newtheorem{theorem}{Theorem}[section]
\newtheorem{definition}[theorem]{Definition}
\newtheorem{lemma}[theorem]{Lemma}
\newtheorem{remark}[theorem]{Remark}

\newtheorem{proposition}[theorem]{Proposition}

\usepackage{bm}

\usepackage{amsmath,amssymb,amsfonts,amsthm}

\usepackage[english]{babel}
\usepackage[babel, english=american]{csquotes}

\usepackage{multirow} 
\usepackage{caption} 
\usepackage{graphicx}

\begin{document}

\title[Asymmetric interface conditions]
{Asymmetric Kedem-Katchalsky boundary conditions for systems with spatial heterogeneities}

\author{Pablo \'Alvarez-Caudevilla, Cristina Br\"{a}ndle, and Fermín González-Pereiro}


\date{\today}


\begin{abstract}
This work investigates a model describing the interaction of two species inhabiting separate but adjacent areas. These populations are governed by a system of equations that account for spatial variations in growth rates and the effects of crowding. A key feature is the presence of  areas within each domain where resources are unlimited and crowding effects are absent. The species interact solely through a common boundary interface, which is modeled by asymmetric Kedem-Katchalsky boundary conditions.
The paper provides existence, non-existence, and behavior of positive solutions for the system. It is shown that a unique positive population distribution exists when one of the growth rate parameters falls within a specific range defined by two critical values. One of these critical values represents a bifurcation point where the population can emerge from extinction, while the other is determined by the characteristics of the refuge areas. The study also examines how the populations behave as the growth parameter approaches the upper critical value. This analysis reveals the phenomena of  non-simultaneous blow-up, where one population component can grow infinitely large within its refuge zone while the other remains bounded.
\end{abstract}

\maketitle

\section{Introduction}

In this work we are going to analyse the interaction of two species that live separately in two subdomains $\Omega_{1}$ and $\Omega_{2}$ of $\Omega \subset \mathbb{R}^{N}$, being $\Omega$ bounded and such that $\Omega = \Omega_{1} \cup \Omega_{2}$ with $\partial\Omega_1={\partial\Omega\setminus\partial\Omega_2}$, (see Figure \ref{figure_domain} for a precise illustration of a domain example).  If we denote by $u_1,u_2$ the densities of these populations we have that $u_i > 0$ in $\Omega_i$ and $u_i = 0$ in $\Omega \setminus \Omega_i$, for $i=1,2$. Under these assumptions we study the following system
\begin{equation*}
    \begin{cases}
        - \Delta u_1 = \lambda m_{1}(x)u_{1} - a_{1}(x)u_{1}^{p} \qquad \text{in} \ \Omega_{1},\\
        - \Delta u_2 = \lambda m_{2}(x)u_{2} - a_{2}(x)u_{2}^{p} \qquad \text{in} \ \Omega_{2},
    \end{cases}
\end{equation*}
where $\lambda$ is a real parameter,  $p > 1$, and the functions $m_i\in L^\infty(\Omega_i)$ are assumed to be strictly positive in each domain, so that
$\lambda m_i(x)$ stands for the intrinsic growth rate of the population.
 Moreover, the coefficients $a_{i}(x)$ are two non negative H\"{o}lder continuous functions, i.e.,  $a_i\in C^{0,\eta}(\overline{\Omega}_i)$, for some $\eta \in (0,1)$,  that measure the crowding effects of the populations in the corresponding subdomains $\Omega_i$. Indeed, we consider that there are regions inside $\Omega_i$ such that the populations $u_i$ have unlimited resources, and grow without limits, allowing $a_i=0$ in some open subdomains of $\Omega_i$, respectively. Hence, we define the non-empty subsets
 $$\Omega_{0,1}:=\{ x \in \Omega_1: a_1(x)=0\} \quad \text{and} \quad \Omega_{0,2}:=\{ x \in \Omega_2: a_2(x)=0\},$$
 so that $\overline{\Omega}_{0,i} \subset \Omega_i$, while $a_i$ remain positive in the rest of the subdomain $\Omega_i$.
Finally, as for the boundary conditions, we consider an inner boundary, $\Gamma=\partial\Omega_1$, that serves as the interface across which the species interact, so that
 \begin{equation*}
    \begin{cases}
    \begin{aligned}
        &\frac{\partial u_1}{\partial \boldsymbol{\nu_1} } = \gamma_1(u_2 - u_1), \quad \text{and} \quad \frac{\partial u_2}{\partial \boldsymbol{\nu_2}} = \gamma_2(u_1 - u_2) &&\text{on} \ \Gamma,\\
               &u_2 = 0 \quad &&\text{on} \ \partial\Omega=\partial \Omega_2 \setminus \Gamma,
        \end{aligned}
    \end{cases}
\end{equation*}
where $\boldsymbol{\nu_i}$ are the unitary outward normal vector which points outward with respect to $\Omega_i$, respectively, with $\boldsymbol{ \nu}:=\boldsymbol{ \nu_1} = - \boldsymbol{\nu_2}$ on $\Gamma$ and $\gamma_i > 0$. Moreover, without loss of generality, we will assume throughout the paper that
$$
\gamma_2\geq \gamma_1.
$$

 \begin{figure}[ht!]
        \centering        \captionsetup{format=hang,name=Fig.,singlelinecheck=off,labelsep=period,labelfont=small,font=small,justification=centering}
        \includegraphics[scale=0.3]{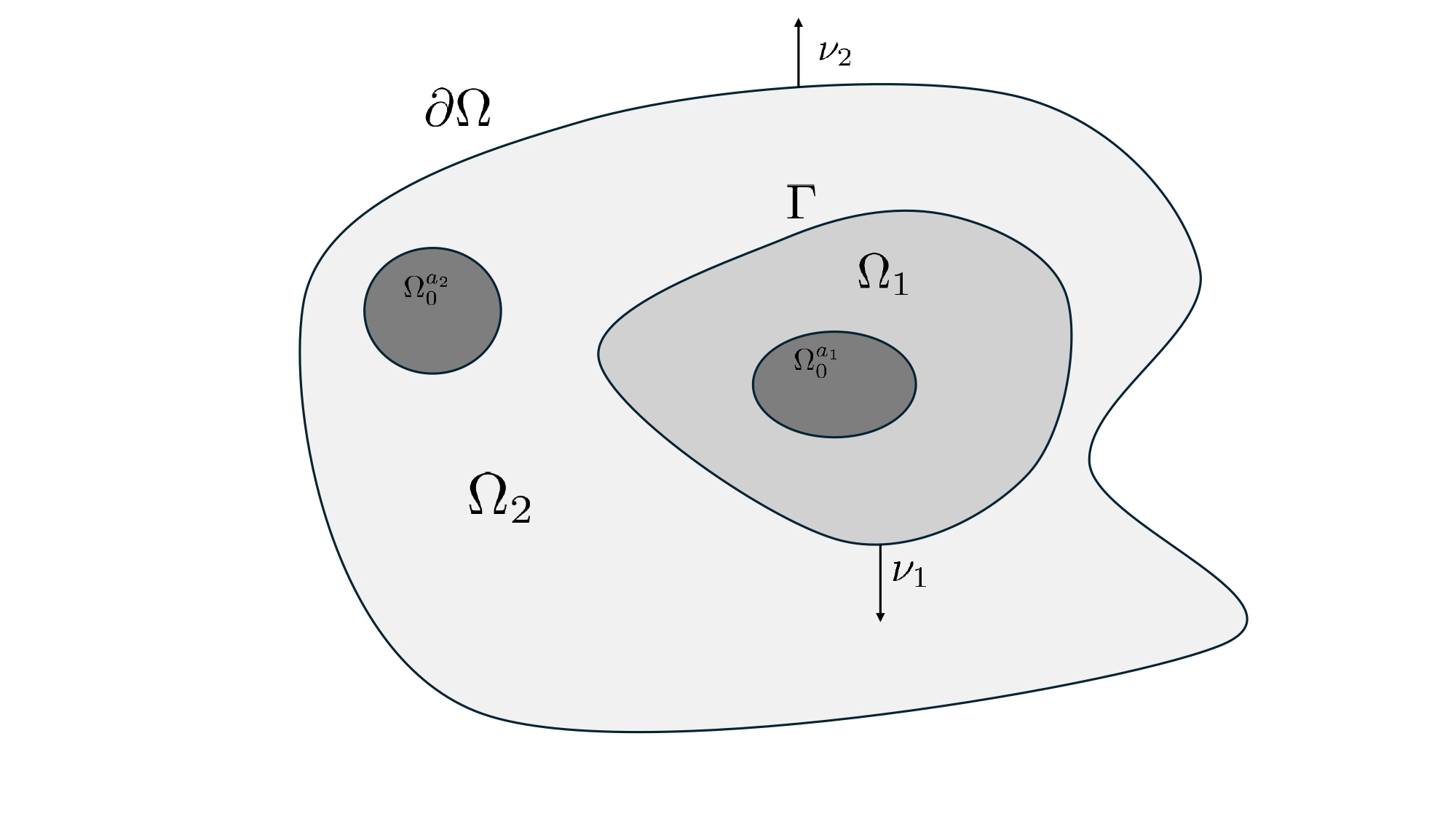}
        \caption{A possible configuration of the domain $\Omega=\Omega_1\cup \Omega_2$.}
        \label{figure_domain}
    \end{figure}

This problem is a general model in population dynamics where both species interact through the common border $\Gamma$, that allows permeation from one part of the domain to the other.  The interface boundary condition are also called Kedem-Katchalsky boundary condition, see~\cite{KedemK} and was introduced in 1961 in a thermodynamic setting. In our particular situation, parameters $\gamma_i$ measure the strength of the membrane depending on the direction the flux is going through. Therefore, the interaction between species occurs only by an interchange of flux, which remains proportional to the density change by a rate $1/\gamma_i$ across the common boundary $\Gamma$.   Thus, the smaller $\gamma_i$, the stronger the barrier is in the corresponding direction, and there is less transmission of that population. In fact, if $\gamma_i \to 0$ for some $i=1,2$ there is no transmission of that population across $\Gamma$. Note that we might have transmission from one domain to the other but not in the other direction, a case that we do not consider here.

From a biological point of view, the use of Kedem-Katchalsky interface conditions goes back to~\cite{Quarteroni}, where the authors studied
the dynamics of the solute in the vessel and in the arterial wall. Since then,  biological applications of membrane problems have increased and have been used  to describe phenomena such as tumour invasion, transport of molecules through the
cell/nucleus membrane, cell polarisation and cell division or  genetics, see~\cite{CiavolellaPaerthame} and references therein.
As far as the study of the mathematical properties is concerned, we were only aware of a few works: In~\cite{WangSu} a semi-linear parabolic problem is considered. The authors investigate the effects of the barrier on the global
dynamics and on the existence, stability, and profile of spatially non-constant equilibria. A similar problem was previously considered in~\cite{Chen2001} and~\cite{Chen2006}.  His analysis was mainly in the framework
of the existence of weak solutions of parabolic and elliptic differential equations with barrier boundary conditions. The author establishes a new comparison principle, the global existence of solutions, and sufficient conditions of stability and instability of equilibria. He shows, in particular, that the stability of equilibrium changes as the barrier permeability changes through a critical value. Very recently, Ciavolella and Perthame~\cite{CiavolellaPaerthame} and~\cite{Ciavolella}  adapted the well-known $L^1$-theory for parabolic reaction-diffusion systems to the membrane boundary conditions case and proved several regularity results. In all these previous problems, the permeability is symmetric, $\gamma_1=\gamma_2$, which makes them self-adjoint and implies the continuity of the flux across the interface.
Finally, in~\cite{Suarezetal} a
semi-linear elliptic interface problem is studied. The authors analyse the existence and uniqueness of positive solutions, assuming that  the interface condition is not symmetric, allowing different in and out fluxes of the domain. It is worth mentioning that all these references deal with problems that do not consider refuges (that is $\Omega_{0,i}$ are empty)  {However, considering these regions where the populations have unlimited resources, enriches the problems, and makes it more realistic, see \cite{LG1997},~\cite{CantrellCosner1991} or~\cite{CantrellCosnerBook}. Let us also mention that recently, \cite{AlvarezBrandleMolinaSuarez2025}, \cite{Maiaetal2024} have introduced a different approach in the analysis of interface problems  allowing certain asymmetry in the equations by assuming different coefficients $\lambda_i$. This fact enriches the models, making them more realistic, and at the same time, makes the analysis more complex.

\noindent{\sc Organization of the paper.} In Section~\ref{sect:auxiliary}, we set the main results about scalar problems, which will be used in the next sections. In Section~\ref{sect.interface} we state the basic definitions concerning interface systems and collect some preliminary results. Also, in this Section, we establish the properties of interface eigenvalue problems and then study the asymptotic behaviour of a linear heterogeneous spatial interface problem. Finally, in Sections~\ref{sect.existence} and~\ref{sect.asymptotic}, the existence, non-existence, and asymptotic behaviour of solutions will be characterized, in addition to analysing aspects related to the bifurcation of the problem in question.

\section{Auxiliary scalar problems}
\label{sect:auxiliary}

  Consider the domain $D\subset \mathbb{R}^N$ which is regular and bounded, such that
$$
\partial D=\Gamma_1\cup \Gamma_2,
$$
where $\Gamma_1$ and $\Gamma_2$ are disjoint. Figure~\ref{figure_domain} could be a particular
case of such a configuration with $D=\Omega_2$,  $\Gamma\equiv \Gamma_1$ and $\Omega_1 =\emptyset$.
Note that it is possible that some $\Gamma_i=\emptyset$. On the boundary, we might impose different boundary conditions such as Robin-type boundary condition of the form
\begin{equation}
\label{eq:BC_eign_prob}
{\mathcal{B}}\varphi:=
{\partial_{\bf n}} \varphi +g_i \varphi \quad  \text{on $\Gamma_i$,}
\end{equation}
 or mixed Robin-Dirichlet (or Neumann-Dirichlet) conditions:
\begin{equation}
\label{eq:BC_eign_mixed}
{\mathcal{B}}\varphi:=
\begin{cases}
\begin{aligned}
& {\partial_{\bf n}} \varphi +g_1 \varphi \quad  &\text{on $\Gamma_1$,}\\
& \varphi &\text{on $\Gamma_2$,}
\end{aligned}
\end{cases}
\end{equation}
where ${\bf n}$ is a nowhere tangent vector field, $g_i\in C(\Gamma_i)$  is nonnegative, for both types of boundary conditions.
For $c\in L^\infty(D)$ we define the general scalar eigenvalue problem
\begin{equation}
  \label{eq:prbm.generico}
   \begin{cases}
            \begin{aligned}&(-\Delta+c(x))\varphi =\lambda m(x) \varphi,&&\text{in} \quad D, \\
           &{\mathcal{B}}\varphi=0, &&\text{on} \quad \partial D,
\end{aligned}\end{cases}
\end{equation}
and denote its principal eigenvalue, which depends on the boundary conditions,  as
$$
\lambda_m^D[-\Delta+c(x); \mathcal{B}],\quad\text{or}\quad \lambda_m^D[-\Delta+c(x);\mathcal{N}+g_1;\mathcal{N}+g_2], \quad\text{or}\quad \lambda_m^D[-\Delta+c(x);\mathcal{N}+g_1;\mathcal{D}] .$$
Recall that the principal eigenvalue is the only eigenvalue $\lambda$ for which there is a positive eigenfunction $\varphi$ verifying~\eqref{eq:prbm.generico}.

Problem~\eqref{eq:prbm.generico} has been widely studied imposing different conditions on $m$ and different boundary conditions, see, for instance,~\cite{CanoLopez2002},~\cite{Hess1983}, or~\cite{KatoBook}. The following holds:
\begin{lemma}
  \label{lemma:escalar.aux}  Let $c,m \in L^\infty(D)$ be such that $m$ is a strictly positive function. Then, Problem~\eqref{eq:prbm.generico} has a unique positive principal eigenvalue  $\lambda_{m}:=\lambda_{m}^D[-\Delta+c(x);\mathcal{B}]$ corresponding to the principal eigenfunction $\varphi\in H^2(D)$, under such Robin boundary conditions \eqref{eq:BC_eign_mixed} and satisfying  $\varphi(x)>0$ when $x\in D\cup\Gamma_1$ and $\partial_{\bf \nu}\varphi(x)<0$ if $x\in \Gamma_2$.
\end{lemma}

In the next proposition, we collect some properties concerning the principal eigenvalue (see for example \cite{CanoLopez2002} for further details about their proofs), $\lambda_1^D[-\Delta+c(x); \mathcal{B}]$, of~\eqref{eq:prbm.generico} when no weight is considered, that is $m\equiv 1$. We will use these properties in the sequel.

\begin{proposition}[\cite{CanoLopez2002}] Let $\lambda_1^D[-\Delta+c(x); \mathcal{B}]$ be the principal eigenvalue of~\eqref{eq:prbm.generico}.
Then
 \begin{enumerate}
        \item[\emph{ 1)}] The map $c \in L^{\infty}(D) \mapsto \lambda_1^D[-\Delta+c(x); \mathcal{B}]$ is continuous and increasing.
        \item[\emph{ 2)}] It holds that $$\lambda_1^D[-\Delta + c; \mathcal{N}+{g_1};\mathcal{N}+{g_2}] < \lambda_1^D[-\Delta + c; \mathcal{N}+g_1;\mathcal{D}].$$
        \item[\emph{ 3)}]Let $D_0$ be a proper subdomain of class $\mathcal{C}^2$ of $D$. Then, $$\lambda_1^D[-\Delta+c(x); \mathcal{B}] < \lambda_1^{D_0}[-\Delta+c(x); \mathcal{B}].$$
        \item[\emph{ 4)}] If there exists a regular, positive $\overline{u}$ such that $$-\Delta \overline{u} + c(x) \overline{u} \geq 0 \quad \text{in} \quad D, \quad \mathcal{B}(\overline{u})\geq 0 \quad \text{on} \quad \partial D,$$ and some of these inequalities is strict, then $$\lambda_1^D[-\Delta+c(x); \mathcal{B}]>0.$$
    \end{enumerate}
    \label{prop:eigvals}
\end{proposition}

\begin{remark}
\label{rem:linear.eigen}
  It is straightforward that for every $s \in \mathbb{R}$
$$\lambda_1^D[-\Delta+c(x)+s; \mathcal{B}] = \lambda_1^D[-\Delta+c(x); \mathcal{B}] + s,$$ a property that we will use in the sequel.
\end{remark}

Observe that for the general  eigenvalue problem~\eqref{eq:prbm.generico}
there is an equivalent formulation in terms of the problem
\begin{equation}
  \label{eq:prbm.generico2}  \begin{cases}
            \begin{aligned}&(-\Delta+c(x)-\lambda m(x)) \varphi=\sigma \varphi,&&\text{in} \quad D, \\
           &{\mathcal{B}}\varphi=0, &&\text{on} \quad \partial D,
\end{aligned}\end{cases}
\end{equation}
so that the principal eigenvalue $\lambda_m^D$ of~\eqref{eq:prbm.generico} is the one that corresponds to the eigenvalue $\sigma_1^D[-\Delta+c(x)-\lambda m; \mathcal{B}]=0$ in~\eqref{eq:prbm.generico2}.
In other words, characterizing $\lambda_m^D$ is equivalent to finding zeros for the function
$$\sigma(\lambda)=\sigma_1^D[-\Delta+c(x)-\lambda m; \mathcal{B}],\quad \lambda\in\mathbb{R}.$$

\begin{theorem}[\cite{CanoLopez2002}, \cite{KatoBook}]
Let $c\in L^\infty(D)$ and let $m\in L^\infty(D)$ be a strictly positive function. Then, the map $\lambda\mapsto \sigma(\lambda)$ defined above is differentiable, decreasing, and concave. Moreover,
$$ \lim_{\lambda\to\infty} \sigma(\lambda)=-\infty,\quad \text{and}\quad \lim_{\lambda\to-\infty} \sigma(\lambda)=\infty.$$
\end{theorem}
Observe that this theorem implies that~\eqref{eq:prbm.generico} has a principal eigenvalue $\lambda_m^D[-\Delta+c(x);\mathcal{B}]$, which is unique.
The positivity of the eigenvalue is a direct consequence of \cite[Theorem 12.1]{Am3}, and this unique positive eigenvalue (the principal eigenvalue) is associated with a positive eigenfunction (the principal eigenfunction), which is unique up to a multiplicative constant.

Finally, we analyse a non-linear problem which will be pivotal in the final section of this paper. Let us consider the set $D_0$ so that $\overline{D}_0\subset D$  and the problem
   \begin{equation}
        \begin{cases}
            \begin{aligned}
                & (-\Delta+c(x))u = \lambda m(x) u - a(x) u^p \quad && \text{in }  D, \\
               & {\partial_{\bf n}} u +g_1u=C \quad  &&\text{on $\Gamma_1$,}\\
& u=0 &&\text{on $\Gamma_2$,}
          \end{aligned}
        \end{cases}
   \label{eq:problem_aux}
   \end{equation}
where  $a(x)=0$ if $x\in D_0$ and $a(x)>0$ in $D\setminus \overline D_0$. The constant $C$ is assumed to be non-negative.
Then, we state the next result whose proof can be obtained by adapting the results obtained in \cite{Fraile1996} for such problems as \eqref{eq:problem_aux}, though with non-homogeneous boundary conditions.
\begin{lemma}
\label{lemma:auxiliary.complete}
  Let $\lambda\in\mathbb{R}$ and $c\in L^\infty(D)$. Problem~\eqref{eq:problem_aux} has a unique positive solution $u \in C^{2,\eta}(D)$  if and only if
  $$\lambda_m^D[-\Delta+c(x),\mathcal{N}+g_1;\mathcal{D}]<\lambda<\lambda_m^{D_0}[-\Delta+c(x);\mathcal{D}].$$
\end{lemma}

Furthermore, for the nonlinear scalar problem \eqref{eq:problem_aux}
 we say that  $\overline{u}$ is a supersolution
 if $\overline{u} \in C^{2}(D) \cap C^{1}(\overline{D})$ and
 \begin{equation}
        \begin{cases}
            \begin{aligned}
                & (-\Delta+c(x))\overline{u} \geq \lambda m(x) \overline{u} - a(x) \overline{u}^p \quad && \text{in } D, \\
               & {\partial_{\bf n}} \overline{u} +g_1 \overline{u}-C\geq 0 \quad  &&\text{on $\Gamma_1$,}\\
&  \overline{u}\geq 0 &&\text{on $\Gamma_2$,}
          \end{aligned}
        \end{cases}
        \label{eq:super-sub-def}
   \end{equation}
Equivalently, we define a subsolution with the reverse inequalities $\leq$. The next lemma provides a comparison result for sub/and super solutions in $D\setminus \overline D_0$

\begin{lemma}    \label{aux_lemma_minimal}
    Let $\underline{u}, \overline{u}\in C^2(\overline D\setminus \overline D_0)$ be a pair of sub and super solutions in $D\setminus \overline D_0$ as defined by~\eqref{eq:super-sub-def}. If
$$\limsup_{{\rm dist}(x,\partial D_0)\to 0} (\overline{u}-\underline{u}) \geq 0,$$
     then $\overline{u} \geq \underline{u}$ in ${D} \setminus \overline {D}_0$.
\end{lemma}

\begin{proof} The proof is based on an argument shown in \cite{DuHuang1999}, see also~\cite{MV}. Let us argue by contradiction considering for $\epsilon_1>\epsilon_2>0$  the set
$$D_{+}(\epsilon_1,\epsilon_2)=\{ x \in D \setminus \overline{D}_{0}\,:\, \underline{u}(x)+\epsilon_2 > \overline{u}(x)+\epsilon_1\}.$$
Define also the functions
$\underline{w}, \overline{w}\in H^1(D \setminus \overline{D}_{0})\cap L^\infty(D \setminus \overline{D}_{0})$, as
    \begin{equation*}
          \overline{w} = ( \overline{u} + \epsilon_1)^{-1} \left[ (\underline{u} + \epsilon_2)^2 - (\overline{u}+\epsilon_1)^2 \right]^+\quad\text{and}\quad
      \underline{w} = ( \underline{u} + \epsilon_2)^{-1} \left[ (\underline{u} + \epsilon_2)^2 - (\overline{u}+\epsilon_1)^2 \right]^+,
                \end{equation*}
    where $f^+=\max\{f,0\}$.

First take two nonnegative test functions $\overline{v},\underline{v}\in C^2({D} \setminus \overline{D}_0)$, vanishing near $\partial D_0$ and $\Gamma_2$. Now, multiplying the inequalities in~\eqref{eq:super-sub-def} for sub and supersolution by $\overline{v}$ and $\underline{v}$, respectively, integrating by parts in $D \setminus \overline{D}_{0}$ and subtracting both expressions, we arrive at
    \begin{equation}
    \label{ineq_proof_aux_lemma_minimal}
    \begin{aligned}
        -&\int_{D \setminus \overline{D}_{0}} \left(\nabla \underline{u} \nabla \underline{v} - \nabla \overline{u} \nabla \overline{v}+c(x)(\underline{v} \underline{u}-\overline{v}\overline{u})\right) - \int_{\Gamma_1}  g_1(x)(\underline{v} \underline{u}-\overline{v} \overline{u}) -C \int_{\Gamma_1}  (\overline{v}-\underline{v}) \\
        &\geq \int_{D \setminus \overline{D}_{0}} a(x)(\underline{u}^p \underline{v}-\overline{u}^p \overline{v}) + \lambda \int_{D \setminus \overline{D}_{0}} m(x)( \overline{u}\overline{v} - \underline{u}\underline{v}).
        \end{aligned}
    \end{equation}
Now, using a regularizing argument we can approximate $\underline{w}$ and $\overline{w}$ in the the $H^1\cap L^\infty$ norm on $D \setminus \overline{D}_{0}$
    by $C^2$ continuous vanishing near $D_0$ and $\Gamma_2$ and we can replace $\underline{w}$ and $\overline{w}$
in the expression~\eqref{ineq_proof_aux_lemma_minimal}. Next we analyse now each of the terms in the expression~\eqref{ineq_proof_aux_lemma_minimal}.

First observe that, due to their definition,  $\overline{w}, \underline{w}\equiv 0$ outside the set $D_{+}(\epsilon_1,\epsilon_2)$ and $\overline{w}>\underline{w}$ in $D_{+}(\epsilon_1,\epsilon_2)$. Hence,
   \begin{equation*}
             -\int_{D \setminus \overline{D}_{0}} (\nabla \underline{u} \nabla \underline{w} - \nabla \overline{u} \nabla \overline{w})=-\int_{D_{+}(\epsilon_1,\epsilon_2)} \left( \left| \nabla \underline{u} - \frac{\underline{u}+\epsilon_2}{\overline{u}+\epsilon_1}\nabla \overline{u} \right|^2 +
    \left| \nabla \overline{u} - \frac{\overline{u}+\epsilon_1}{\underline{u}+\epsilon_2}\nabla \underline{u} \right|^2 \right)\leq 0,
    \end{equation*}
and  moreover
 $$-C \int_{\Gamma_1}  (\overline{w}-\underline{w})\leq 0.$$
On the other hand, since $0<\epsilon_2<\epsilon_1$, when $\epsilon_1\to 0$ it is straightforward to see that the terms
   \begin{equation*}
        -\int_{D_{+}(\epsilon_1,\epsilon_2)}c(x)(\underline{w} \underline{u}-\overline{w}\overline{u}),\quad  -  \int_{\Gamma_1} g_1(x) (\underline{w} \underline{u}-\overline{w} \overline{u}), \quad  \lambda \int_{D_{+}(\epsilon_1,\epsilon_2)} m(x)( \overline{u}\overline{w} - \underline{u}\underline{w}).
    \end{equation*}
converge to $0$ and that
$$
\int_{D_{+}(\epsilon_1,\epsilon_2)} a(x)(\underline{u}^p \underline{w}-\overline{u}^p \overline{w}) \to \int_{D_{+}(0,0)} a(x) (\underline{u}^{p-1} - \overline{u}^{p-1})(\underline{u}^2-\overline{u}^2).
$$
This last integral is strictly positive, unless the set $D_{+}(0,0)$ is empty, in which case it vanishes. Hence, gathering all the terms we get
$$
\begin{aligned}
0&\geq-\int_{D_{+}(\epsilon_1,\epsilon_2)} \left( \left| \nabla \underline{u} - \frac{\underline{u}+\epsilon_2}{\overline{u}+\epsilon_1}\nabla \overline{u} \right|^2 +
    \left| \nabla \overline{u} - \frac{\overline{u}+\epsilon_1}{\underline{u}+\epsilon_2}\nabla \underline{u} \right|^2 \right)-C \int_{\Gamma_1}  (\overline{w}-\underline{w})\\
&\geq \int_{D_{+}(0,0)} a(x) (\underline{u}^{p-1} - \overline{u}^{p-1})(\underline{u}^2-\overline{u}^2)>0,
\end{aligned}$$
which is a contradiction. Consequently, $\underline{u} \leq \overline{u}$ in $D \setminus \overline{D}_{0}$.
\end{proof}

\section{Interface problems}
\label{sect.interface}

In this section, we gather several preliminary results in relation to interface problems. However, even though the aim of this paper is to analyze interface nonlinear problems of the form
$$
\left\{\begin{array}
  {ll}
  - \Delta u_i = \lambda m_{i}(x)u_{i} - a_{i}(x)u_{i}^{p}, \quad &x\in\Omega_{i},\\
   \mathcal{I}(\mathbf{u})=0,&x\in\Gamma\cup \partial \Omega,
\end{array}\right.
$$
an important part of the analysis is based on studying the linear and eigenvalue counterpart problems, for which we will show several important results in this section.
To do so, we first fix the notation that we will use throughout the following sections. First note that when no confusion arises, we will omit $i=1,2$ and write the problem as above,
or equivalently
\begin{equation}
 \label{eq:original_problem}
\left\{\begin{array}
  {ll} -\Delta \mathbf{u}=\lambda {\mathbf{mu}}-{\bf a u}^p,\quad &x\in\Omega,\\
  \mathcal{I}(\mathbf{u})=0,&x\in\Gamma\cup \partial \Omega.
\end{array}\right.
\end{equation}
Moreover, the boundary conditions $\mathcal{I}({\bf u}) =0$ stands for
\begin{equation}
  \label{eq:conditions}
  \frac{\partial u_1}{\partial \boldsymbol{\nu_1}} = \gamma_1(u_2 - u_1) \quad \hbox{and}\quad \frac{\partial u_2}{\partial \boldsymbol{\nu_2}} = \gamma_2(u_1 - u_2) \quad \text{on} \ \Gamma,\quad\text{and}\quad u_2 =  0 \quad \text{on} \ \partial \Omega.
\end{equation}
 Here $\mathbf{u}=(u_1,u_2)$ with each $u_i$ defined in $\Omega_i$ and $\mathbf{u^p}= (u_1^p,u_2^p)$. Similarly for $\mathbf{m}$ and $\mathbf{a}$.

Now, before showing the results for the auxiliary linear associated problems, first let us fixed the notation and the functional spaces that we are going to use.

To perform the analysis in this work we denote $\mathcal{C}^\eta$, $\mathcal{L}^p$, $\mathcal{H}^1$ and  $\mathcal{W}^{2,p}$ as  the product spaces $C^\eta(\Omega_1)\times C^\eta(\Omega_2)$, $L^p(\Omega_1)\times L^p(\Omega_2)$,  $H^1(\Omega_1)\times H^1(\Omega_2)$ and $W^{2,p}(\Omega_1)\times W^{2,p}(\Omega_2)$, respectively and the norm of a function $\bf{w}$ is defined as the sum of the norms of $w_i$ in the respective spaces. To impose the boundary conditions on the functional spaces we will define, for instance,
$$
\mathcal{C}_\Gamma(\Omega):=\{{\bf u}\in C(\overline\Omega_1)\times C(\overline\Omega_2) : \mathcal{I}({\bf u}) =0\}
$$
We define in a similar way the spaces of continuously differentiable and H\"{o}lder continuous functions $\mc{C}_{\Gamma}^{1}$, $\mc{C}_{\Gamma}^{1,\eta}$,  $\dots$
for some $\eta\in (0,1]$,  $\mathcal{L}_\Gamma^p$, $\mathcal{H}_\Gamma^1$ and  $\mathcal{W}_\Gamma^{2,p}, \dots$

Finally, we will write ${\bf u} \geq 0$ in $\Omega$ if $u_i \geq 0$ in $\Omega_i$. Analogously, we write ${\bf u} >0$ in $\Omega$ if $u_i > 0$ in $\Omega_i$ and $\bf{u} \neq 0$ in $\Omega$ if $u_i \neq 0 $ in a subset of positive measure of $\Omega_i$.

\subsection{Interface linear problems}
\label{sect.preliminaries}
Let us consider the asymmetric linear interface problem
\begin{equation}
\left\{\begin{array}
  {ll}
    (- \Delta  + {\bf c}(x)){\bf u} = {\bf f}(x), \quad &\text{in} \ \Omega
    \\
\mathcal{I}(\mathbf{u})=0,&\text{on}\  \Gamma\cup \partial \Omega,
\end{array}\right.
        \label{eq:linear}
\end{equation}
such that ${\bf c}(x)=(c_1(x),c_2(x))\in \mathcal{L}^\infty$ and ${\bf f}(x)=(f_1(x),f_2(x))\in \mathcal{L}^2$.
Problem~\eqref{eq:linear} with homogeneous Neumann  boundary condition on the outside boundary $\partial\Omega$,
i.e., $\frac{\partial u_2}{\partial \boldsymbol{\nu}}=0$ on $\partial \Omega$, has been studied in~\cite{Suarezetal} and~\cite{WangSu}.

The bilinear form associated with (\ref{eq:linear}) is $a: \mathcal{H}_\Gamma^1 \times \mathcal{H}_\Gamma^1 \rightarrow \mathbb{R}$ and defined by
\begin{equation}
    a({\bf{u}},{\bf{v}}) := \sum_{i=1}^{2} \left( \int_{\Omega_i} \nabla u_i \cdot \nabla v_i + c_i(x)u_i v_i \right) + \int_{\Gamma} (u_2 - u_1)(\gamma_2 v_2 - \gamma_1 v_1).
    \label{eq:bilineal}
\end{equation}

\begin{definition}
    We say that ${\bf{u}}\in \mathcal{H}_\Gamma^1$ is a weak solution of~\eqref{eq:linear} if
    $$a({\bf{u}},{\bf{v}})=\langle {\bf{f}},{\bf{v}} \rangle _{\mathcal{L}^2},$$
    for all {${\bf{v}} \in \mathcal{H}_\Gamma^1$.}
\end{definition}

\begin{definition}
    We say that ${\bf{u}}$ is a classical solution of~\eqref{eq:linear} if $u_i\in C^2(\Omega_i)\cap C^{1}(\overline\Omega_i)$ and ${\bf u}$ satisfies~\eqref{eq:linear} pointwise.
\end{definition}

We will also use the notation $\mathcal{I}({\bf u}) \succeq 0$ to denote
\begin{equation}
        \frac{\partial u_1}{\partial \boldsymbol{\nu_1}} \geq  \gamma_1(u_2 - u_1), \quad \frac{\partial u_2}{\partial \boldsymbol{\nu_2}} \leq  \gamma_2(u_2 - u_1) \quad \text{on} \ \Gamma, \quad\text{and}\quad u_2 \geq  0 \quad \text{on} \ \partial \Omega,
    \label{conditions_supersolution}
\end{equation}
and $\mathcal{I}({\bf u}) \preceq 0$ for the reverse inequalities.

\begin{definition}
\label{def.subsuper}
    We say that  $\overline{\bf{u}}$ is a supersolution of~\eqref{eq:linear}  if $\overline{u}_i \in C^{2}(\Omega_i) \cap C^{1}(\overline{\Omega}_i)$ and
    \begin{equation*}
        (-\Delta  + {\bf c}(x)){\bf{\overline{u}}} \geq {\bf f}\ \quad \text{in} \ \Omega, \qquad
        \mathcal{I}({\bf{\overline{u}}}) \succeq 0 \quad \text{on} \ \Gamma\cup  \partial \Omega.
    \end{equation*}
    If moreover, any of these inequalities is strict, then we say that ${\bf{u}}$ is a strict supersolution.

\noindent    A subsolution $\underline{\bf{u}}$ is defined analogously by taking $\leq$ and $\preceq$.
\end{definition}

\begin{theorem}
    Let ${\bf c}\in \mathcal{L}^\infty$, ${\bf c}\geq c_0>0$ and ${\bf f}\in\mathcal{L}^2$. Then there exists a unique weak solution ${\bf u}\in\mathcal{H}^1$ of~{\rm\eqref{eq:linear}}. Moreover, by elliptic regularity, if
    ${\bf c}, {\bf f}\in\mathcal{C}^\eta$, $\eta\in(0,1)$ then there exists a unique classical solution  ${\bf u}\in\mathcal{C}^{2,\eta}$ of~{\rm\eqref{eq:linear}}.
\end{theorem}

\begin{proof}
  The proof can be adapted straightforwardly  from~\cite{Suarezetal},~\cite{WangSu} and~\cite{AlvarezBrandle}.
\end{proof}

The first result provides us with a maximum principle.
\begin{lemma}[\cite{WangSu}]
    Let ${\bf u}$ be a classical supersolution such that
    $$
   (- \Delta  + {\bf c}(x)){\bf u} \geq {\bf 0}, \quad \text{in} \ \Omega
    \qquad\hbox{and}\qquad
\mathcal{I}({\bf{\overline{u}}}) \succeq 0, \quad \text{on}\  \Gamma\cup \partial \Omega.
    $$
Then, the following holds:
\begin{enumerate}
        \item[\emph{ 1)}] If $\bf{u} \geq 0$ in $\Omega$, then, either $\bf{u} > 0$ or $\bf{u} \equiv 0$ in $\Omega$.
        \item[\emph{ 2)}]  If  $\bf{c} \geq 0,$ $\bf{c} \neq 0$ in $\Omega$, then, $\bf{u} > 0$ in $\Omega$.
    \end{enumerate}
    \label{lemma:max_princ}
\end{lemma}

\subsection{Interface asymmetric eigenvalue problems}
 In our context, the interface counterpart of~\eqref{eq:prbm.generico} is
\begin{equation}
\label{eq:problem_lambda*} \begin{cases}
    \begin{aligned}
        &(-\Delta+{\bf c(x)}) {\bf \Phi} =\lambda {\bf m(x)\Phi} \quad &&\text{in} \ \Omega, \\
         &\mathcal{I}({\bf \Phi}) = 0 \quad &&\text{on} \ \Gamma \cup \partial \Omega.
             \end{aligned}
    \end{cases}
\end{equation}
 This problem was addressed in~\cite{AlvarezBrandle} under symmetric interface conditions. In~\cite{Suarezetal} and \cite{WangSu} the authors consider Neumann boundary conditions on $\partial \Omega$ and ${\bf m}$, eventually, changing sign. As before, first of all we state the basic results for the interface eigenvalue problem when no weight is considered:
\begin{equation}
    \begin{cases}
    \begin{aligned}
        & (-\Delta + {\bf c}(x) ){\bf \Phi} = \lambda {\bf \Phi} \quad &&\text{in} \ \Omega, \\
        &\mathcal{I}({\bf \Phi}) = 0 \quad &&\text{on} \ \Gamma\cup \partial \Omega.
        \end{aligned}
    \end{cases}
    \label{prob:general}
\end{equation}
The proofs of these results are similar to the ones that can be found in~\cite{AlvarezBrandle},~\cite{Suarezetal}, and~\cite{WangSu} and we omit the details.

\begin{theorem}
    Let ${\bf c} \in  \mathcal{L}^{\infty},$ respectively ${\bf c} \in \mathcal{C}^{\eta}$, with $\eta\in(0,1)$. Then,  Problem~{\rm \eqref{prob:general}} has
    a principal eigenvalue $\Lambda_1^\Omega[\Delta +{\bf c}, \mathcal{I}]$, which is real and simple.  Moreover, the associated eigenfunction ${\bf \Phi}\in\mathcal{W}^{2,p},p>1$, respectively ${\bf \Phi}\in \mathcal{C}^{2,\eta}$,  is strictly positive in $\Omega$.  \label{theo:eigenval}
\end{theorem}

The next result provides us with a characterization of the strong maximum principle, equivalent to the result  given \cite{lopez1996maximum} but particularised to the interface system.
\begin{proposition}
\label{prop:characterization.interface}
    Let $\Lambda_1^\Omega[-\Delta+{\bf c}; \mathcal{I}]$ be the principal eigenvalue of~\eqref{prob:general}. The following are equivalent:
    \begin{enumerate}
        \item[\emph{1)}] $\Lambda_1^\Omega[\Delta + {\bf c};\mathcal{I}]>0.$
        \item[\emph{2)}]  There exists a strict positive supersolution $\bf{\overline{u}}$ of~\eqref{eq:linear} with ${\bf f}\equiv 0$.
        \item[\emph{3)}]  Given ${\bf f} \in \mathcal{C}^{1}, {\bf f} \geq 0, {\bf f} \neq 0$ in $\Omega$, there exists a unique positive solution $\bf{u}$ of~\eqref{eq:linear}.
    \end{enumerate}
\end{proposition}

The principal eigenvalue of~\eqref{prob:general} satisfies the following result related to monotonicity and convergence similar to Lemma~\ref{prop:eigvals}.
\begin{proposition}
\label{prop.properties.Lambda}
Let $\Lambda_1^\Omega[-\Delta+{\bf c}; \mathcal{I}]$ be the principal eigenvalue of~\eqref{prob:general}.
    \begin{enumerate}
        \item[\emph{1)}] The map ${\bf c} \in \mathcal{L}^{\infty} \mapsto \Lambda_1^\Omega[-\Delta+{\bf c}; \mathcal{I}]$ is continuous and increasing.
        \item[\emph{2)}] Suppose that ${\bf c}_n \rightarrow \bf{c}$ in $\mathcal{L}^{\infty}$, then $\Lambda_1^\Omega[-\Delta+{\bf c_n}; \mathcal{I}] \rightarrow \Lambda_1^\Omega[-\Delta+{\bf c}; \mathcal{I}].$
        \item[\emph{3)}] It holds that $$\Lambda_1^\Omega[-\Delta + {\bf c},\mathcal{I}] < \min \{ \lambda_1^{\Omega_1} [ -\Delta + c_1; \mathcal{N} + \gamma_1],  \lambda_1^{\Omega_2} [ -\Delta + c_2; \mathcal{N} + \gamma_2,\mathcal{D}] \}.$$
    \end{enumerate}
\end{proposition}
\begin{proof} Let $\boldsymbol{\Phi}=(\phi_1,\phi_2)$ be a positive eigenfunction associated to $\Lambda_1:=\Lambda_1^\Omega[-\Delta+{\bf c}; \mathcal{I}]$.

\noindent 1) Assume that ${\bf c}\leq {\bf d}$. Then it is straigthforward to see that
$$
      -\Delta \boldsymbol{\Phi} + ({\bf d} - \Lambda_1 ) \boldsymbol{\Phi} = ( {\bf d} - {\bf c} ) \boldsymbol{\Phi} \geq 0 \quad \text{in} \ \Omega,
$$
and $$\mathcal{I}(\boldsymbol{\Phi}) =0 \quad \text{on} \ \Gamma\cup \partial \Omega.$$
Hence, $\boldsymbol{\Phi}$ is a positive supersolution of $(-\Delta + {\bf d} - \Lambda_1; \mathcal{I}$) and therefore, due to Proposition~\ref{prop:characterization.interface},
$$\Lambda_1^\Omega[-\Delta + {\bf d} - \Lambda_1]  = \Lambda_1^\Omega [-\Delta + {\bf d}] - \Lambda_1  \geq 0.  $$

\noindent 2) Since  ${\bf c}_n \rightarrow \bf{c}$ in $\Omega$, for any $\epsilon>0$ there exists $N \in \mathbb{N}$ such that ${\bf c} - \epsilon \leq {\bf c}_{n} \leq {\bf c} + \epsilon$, for $n \geq N$. Applying the previous item we obtain
$$
\Lambda_1^\Omega[-\Delta+({\bf c} - \epsilon)] \leq \Lambda_1^\Omega[-\Delta+{\bf c}]
\leq \Lambda_1^\Omega[-\Delta+({\bf c} + \epsilon)],
$$
which by Remark~\ref{rem:linear.eigen} implies
$$
\Lambda_1^\Omega[-\Delta+{\bf c}] - \epsilon \leq \Lambda_1^\Omega[-\Delta+{\bf c}]
\leq \Lambda_1^\Omega[-\Delta+{\bf c}] + \epsilon,
$$
and hence the convergence.

\noindent 3)  Consider the scalar problem
    \begin{equation}
    \label{eq:auxiliar.bound.eigenvalue}
        \begin{cases}\begin{aligned}
            && -\Delta \phi + (c_1 - \Lambda_1) \phi = 0 \quad &\text{in} \ \Omega_1, \\ && \frac{\partial \phi}{\partial \boldsymbol{\nu_1}} + \gamma_1 \phi = 0 \quad &\text{on} \ \Gamma.
        \end{aligned}\end{cases}
    \end{equation}
    It is clear that $\phi_1$ is a positive supersolution to~\eqref{eq:auxiliar.bound.eigenvalue}, since it verifies
    $$
    \frac{\partial \phi_1}{\partial \boldsymbol{\nu_1}} + \gamma_1 \phi_1 = \gamma_1 \phi_2 >0,\quad  \text{on} \ \Gamma.
    $$
 Hence, by Proposition \ref{prop:eigvals}  we get $\lambda_1^{\Omega_1}[-\Delta + c_1 -  \Lambda_1; N + \gamma_1] >0,$ or equivalently
 $$\lambda_1^{\Omega_1}[-\Delta + c_1 ; N + \gamma_1] >\Lambda_1.$$
 The proof for $\lambda_1^{\Omega_2} [ -\Delta + c_2; \mathcal{N} + \gamma_2,\mathcal{D}]$ is analogous.
\end{proof}

Note that, as shown for similar problems above, the problem of analyzing the existence of
 a principal eigenvalue of problem~\eqref{eq:problem_lambda*} is equivalent to finding zeros of the function
 $$
 \Sigma(\lambda)=\Lambda_1^\Omega[-\Delta+{\bf c}-\lambda{\bf m};\mathcal{I}].
 $$
\begin{theorem}
\label{thm:unique.lambda.eignefunction}
 Let ${\bf c}\in \mathcal{L}^\infty$ and let ${\bf m}\in \mathcal{L}^\infty$ be a strictly positive function. Problem~\eqref{eq:problem_lambda*} has a principal eigenvalue, $\Lambda_m^\Omega[-\Delta+{\bf c};\mathcal{I}]$. Moreover, $\Lambda_m^\Omega[-\Delta+{\bf c};\mathcal{I}]$ is real, simple (in the sense of
multiplicities), and strictly positive with the associated eigenfunction ${\bf \Phi} \in \mathcal{H}^1$ strictly positive in $\Omega$.
\end{theorem}

\begin{remark}
    As in Theorem~{\rm \ref{theo:eigenval}}, imposing extra regularity, ${\bf c\in \mathcal{C}^\eta}$ and ${\bf m}\in\mathcal{C}^\eta$ yields ${\bf \Phi\in \mathcal{C}^2}$; see~\cite{WangSu} for further details.
\end{remark}

\begin{proof}
    The proof is a direct corollary of the fact that the function $\Sigma(\lambda)$ is differentiable, decreasing, concave, and satisfies
    $$
        \lim_{\lambda\to\infty} \Sigma(\lambda) =-\infty,\quad \lim_{\lambda\to-\infty} \Sigma(\lambda)
        =+\infty,
     $$
     see~\cite{Suarezetal}. In that situation, it is clear that there is a unique $\lambda_0$ such that $\Sigma(\lambda_0)=0$. The positivity of  $\Lambda_m^\Omega[-\Delta+{\bf c};\mathcal{I}]$ follows from the maximum principle. Finally, the regularity of ${\bf \Phi}$  can be shown using standards arguments of elliptic equations (cf.~\cite{WangSu})
\end{proof}

\subsection{Asymptotic behaviour}
\label{sec:asymptotic_behaviour}

Given a real parameter $\alpha$, we now consider problem~\eqref{eq:problem_lambda*}  with ${\bf c}(x)=\alpha{\bf a}(x)$. Recall that ${\bf a}$ is the function that measures the crowding effect and $a_i=0$ in $\Omega_{0,i}$. We study the limit behaviour, as $\alpha$ grows to infinity, of the linear weighted elliptic eigenvalue problem
\begin{equation}
    \begin{cases}
    \begin{aligned}
        &(-\Delta + \alpha {\bf a} ) {\bf{\Phi_{\alpha}}} = \Lambda_\alpha {\bf m}  {\bf \Phi_{\alpha}} \quad &&\text{in} \ \Omega, \\
         &\mathcal{I}({\bf \Phi}_\alpha) = 0 \quad &&\text{on} \ \Gamma \cup \partial \Omega.
             \end{aligned}
    \end{cases}
    \label{eq:weighted_elliptic_vector}
\end{equation}
This analysis will be used later in this paper to establish the existence of positive solutions for the nonlinear interface problem \eqref{eq:original_problem}.

For a fixed $\alpha$ we will denote by $\Lambda_\alpha$ the principal eigenvalue and ${\bf \Phi}_\alpha\in \mathcal{H}_\Gamma^1$ the corresponding principal  eigenfunction of~\eqref{eq:weighted_elliptic_vector}, or just $(\Lambda_\alpha, {\bf \Phi}_\alpha)$. Recall, thanks to Theorem~\ref{thm:unique.lambda.eignefunction}, that since $\phi_{i,\alpha} \geq 0$ we can conclude that $\Lambda_\alpha \geq 0$. Moreover, since the principal eigenfunction ${\bf \Phi}_\alpha$ is unique up to a multiplicative constant, we can assume that  it is normalised, so that
\begin{equation}
    \int_{\Omega_1} m_1(x) \varphi_{1,\alpha}^2 + \int_{\Omega_2} m_2(x) \varphi_{2,\alpha}^2 = 1.
    \label{eq:normalization}
\end{equation}
Under these assumptions in the sequel we will prove that there is a limit problem with solution $(\lambda_\infty,{\bf \Phi}_\infty)$, such that $(\Lambda_\alpha,{\bf \Phi}_\alpha)$ converges to it as $\alpha\to\infty$.
However, we first need to proof some bounds.

\begin{lemma}
    For each fixed $\alpha>0$, let $(\Lambda_\alpha, \bf{\Phi_\alpha})$ be a solution to~{\rm\eqref{eq:weighted_elliptic_vector}}. Then,
    \begin{equation*}
        \| {\bf{\Phi_\alpha}} \|^2_{\mathcal{H}_\Gamma^1} \leq \frac{\gamma_2}{\gamma_1} \Lambda_\alpha, \quad  \int_{\Gamma} (\varphi_{2,\alpha} - \varphi_{1,\alpha})^2 \leq \frac{1}{\gamma_1} \Lambda_\alpha, \quad \alpha \left( \int_{\Omega_1} a_1 \varphi_{1,\alpha}^2 + \int_{\Omega_2} a_2 \varphi_{2,\alpha}^2 \right) \leq \frac{\gamma_2}{\gamma_1} \Lambda_\alpha.
    \end{equation*}
    \label{lemma:bounds:problem_alpha}
\end{lemma}
\begin{proof}
    Let ${\bf\Psi}_\alpha=(\gamma_2\varphi_{1,\alpha},\gamma_1\varphi_{2,\alpha})$. Multiplying (\ref{eq:weighted_elliptic_vector}) by ${\bf \Psi}_\alpha$ and integrating by parts
we get, for the left-hand side, since $\gamma_2\geq\gamma_1$:
    \begin{equation*}
        \begin{aligned}
                & \gamma_2\int_{\Omega_1} \left| \nabla \varphi_{1,\alpha} \right|^2 - \gamma_2\int_{\Gamma} \varphi_{1,\alpha} \frac{\partial \varphi_{1,\alpha}}{\partial \boldsymbol{\nu_1}} + {\gamma_1}\int_{\Omega_2} \left| \nabla \varphi_{2,\alpha} \right|^2 - {\gamma_1}  \int_{\Gamma} \varphi_{2,\alpha} \frac{\partial \varphi_{2,\alpha}}{\partial \boldsymbol{\nu_2}}
            \\
            &\qquad+ \alpha \left( \gamma_2\int_{\Omega_1} a_1 \varphi_{1,\alpha}^2 + {\gamma_1} \int_{\Omega_2} a_2 \varphi_{2,\alpha}^2  \right)\\
              &= \gamma_2\int_{\Omega_1} \left| \nabla \varphi_{1,\alpha} \right|^2 + {\gamma_1}\int_{\Omega_2} \left| \nabla \varphi_{2,\alpha} \right|^2 + \gamma_1 \gamma_2 \int_{\Gamma} (\varphi_{2,\alpha} - \varphi_{1,\alpha})^2\\
               &\qquad+ \alpha \left( \gamma_2\int_{\Omega_1} a_1 \varphi_{1,\alpha}^2 + \gamma_1\int_{\Omega_2} a_2 \varphi_{2,\alpha}^2  \right)\\
               &\geq \gamma_1\|{\bf{\Phi_\alpha}} \|^2_{\mathcal{H}_\Gamma^1}+ \gamma_1 \gamma_2 \int_{\Gamma} (\varphi_{2,\alpha} - \varphi_{1,\alpha})^2+ \alpha\gamma_1 \left( \int_{\Omega_1} a_1 \varphi_{1,\alpha}^2 + \int_{\Omega_2} a_2 \varphi_{2,\alpha}^2  \right).
        \end{aligned}
    \end{equation*}
    For the right hand, we get
\begin{equation*}
    \Lambda_\alpha \left(\gamma_2 \int_{\Omega_1} m_1 \varphi_{1,\alpha}^2 + {\gamma_1}\int_{\Omega_2} m_2 \varphi_{2,\alpha}^2 \right) \leq \gamma_2 \Lambda_\alpha.
\end{equation*} Thus,
\begin{equation*}
   \|{\bf{\Phi_\alpha}} \|^2_{\mathcal{H}_\Gamma^1}+  \gamma_2 \int_{\Gamma} (\varphi_{2,\alpha} - \varphi_{1,\alpha})^2+ \alpha \left( \int_{\Omega_1} a_1 \varphi_{1,\alpha}^2 + \int_{\Omega_2} a_2 \varphi_{2,\alpha}^2\right)\leq \frac{\gamma_2}{\gamma_1} \Lambda_\alpha.
\end{equation*} and we obtain the desired bounds.
\end{proof}

Next, we prove that there exists a limit for the sequence of eigenvalues as $\alpha$ tends to $\infty$.

\begin{lemma}
\label{lemma.limit.lambda_alpha} For each fixed $\alpha>0$, let $(\Lambda_\alpha, \bf{\Phi_\alpha})$ be a solution to~{\rm\eqref{eq:weighted_elliptic_vector}}. Then, there exists $\ell\in\mathbb{R}$ such that
  $$\lim_{\alpha\to\infty} \Lambda_\alpha=\ell.$$
\end{lemma}

\begin{proof}
First, if $\Lambda_\alpha $ is the eigenvalue of~\eqref{eq:weighted_elliptic_vector} with eigenfunction ${\bf \Phi}_\alpha$, then
\begin{equation}
  \label{eq:sigma.zero}
\Sigma_{1}^{\Omega} [-\Delta  + \alpha {\bf a} - \Lambda_\alpha {\bf m} ;\mathcal{I}] = 0.
\end{equation}
Let $\{\alpha_n\}_{n \in \mathbb{N}}$ be an increasing unbounded sequence and let $\Lambda_n$ be the corresponding principal eigenvalue of~\eqref{eq:weighted_elliptic_vector} and ${\bf \Phi}_n$ the principal eigenfunction.

To find a bound for $\Lambda_n$ we use the monotonicity properties of the principal eigenvalue shown in~\cite{CanoLopez2002} and~\cite{Suarezetal}, see also Proposition~\ref{prop:eigvals}. Observe that $\varphi_{1,n}$ is a strict supersolution to the uncoupled problem
$$
(- \Delta -\Lambda_\alpha m_1+\alpha a_1)\varphi_{1,n}=0,\quad\text{in $\Omega_1$},\qquad \partial_{{\bf \nu}_1} \varphi_{1,n}+\gamma_1\varphi_{1,n}=0,\quad \text{in $\Gamma$}.
$$
Hence,  the principal eigenvalue verifies
$$
\sigma_1^{\Omega_1}[- \Delta -\Lambda_\alpha m_1+\alpha a_1;\mathcal{N}+\gamma_1]>0.
$$
Now, due to the monotonicity with respect boundary data and the domain we have
$$
\begin{aligned}
0&<\sigma_1^{\Omega_1}[- \Delta -\Lambda_\alpha m_1+\alpha a_1;\mathcal{N}+\gamma_1]< \sigma_1^{\Omega_1}[- \Delta -\Lambda_\alpha m_1+\alpha a_1;\mathcal{D}]\\
&<\sigma_1^{\Omega_{0,1}}[- \Delta -\Lambda_\alpha m_1]\leq \sigma_1^{\Omega_{0,1}}[-\Delta;\mathcal{D}]-\Lambda_\alpha\min_{x\in\Omega_1}m_1(x).
\end{aligned}
$$
Analogously we obtain that
$$
0<\sigma_1^{\Omega_2}[- \Delta -\Lambda_\alpha m_2+\alpha a_2;\mathcal{N}+\gamma_2;\mathcal{D}]< \sigma_1^{\Omega_{0,2}}[-\Delta;\mathcal{D}]-\Lambda_\alpha\min_{x\in\Omega_2}m_2(x).
$$
Hence, we conclude that $\Lambda_\alpha$ is bounded form above by
\begin{equation}
\Lambda_\alpha<\min\left\{\frac{\sigma_1^{\Omega_{0,1}}[-\Delta;\mathcal{D}]}{\min_{x\in\Omega_1}m_1(x)};  \frac{\sigma_1^{\Omega_{0,2}}[-\Delta;\mathcal{D}]}{\min_{x\in\Omega_2}m_2(x)}\right\}
 \label{eq:bound_lambda_alpha}
\end{equation}

Finally, notice that thanks to the monotonicity of the principal eigenvalues with respect to the potential, Proposition~\ref{prop.properties.Lambda}, we know that the eigenvalues $\Lambda_\alpha$ are increasing in terms of the parameter $\alpha$. Hence, being increasing and bounded from above, we get that there exists $\ell<\infty$ such that
$
\Lambda_{\alpha_n}\to\ell$ as $n\to\infty$.
\end{proof}

\begin{lemma}
\label{lemma:integral.0} For each fixed $\alpha>0$, let $(\Lambda_\alpha, \bf{\Phi_\alpha})$ be a solution to~{\rm\eqref{eq:weighted_elliptic_vector}}. Then
  \begin{equation*}
    \lim_{\alpha \to \infty}  \int_{\Omega_i} a_i \varphi_{i,\alpha}^2 = 0.
\end{equation*}
\end{lemma}

\begin{proof}
Let $\{\alpha_n\}_{n \in \mathbb{N}}$ be an increasing unbounded sequence and ${\bf \Phi}_{\alpha_n}:={\bf \Phi}_{n}=(\varphi_{1,n}, \varphi_{2,n})$ . It is clear from Lemmas~\ref{lemma:bounds:problem_alpha} and~\ref{lemma.limit.lambda_alpha} that, as $n\to\infty$,
     \begin{equation*}
    \lim_{n \to \infty} \left( \int_{\Omega_1} a_1 \varphi_{1,n}^2 + \int_{\Omega_2} a_2 \varphi_{2,n}^2 \right) = 0.
\end{equation*}
 Moreover,  since the sets $\Omega_i$ are disjoint, we can state the result.

\end{proof}

Having introduced the above, we can state the following convergence theorem.

\begin{lemma}
    For each fixed $\alpha>0$, let $(\Lambda_\alpha, \bf{\Phi_\alpha})$ be a solution to~{\rm\eqref{eq:weighted_elliptic_vector}}.  Then, there exists $\Phi\in \mathcal{H}_\Gamma^1$,  with $\Phi\equiv 0$ in $\Omega\setminus(\Omega_{0,1}\cup\Omega_{0,2})$ and such that
    $$
        {\bf \Phi}_\alpha\to \Phi,\quad \text{in}\quad \mathcal{H}_\Gamma^1,$$
        as $\alpha\to\infty$.
\label{lemma:convergence.phi}
\end{lemma}

\begin{proof}
  Let $\{\alpha_n\}_{n \in \mathbb{N}}$ be an increasing unbounded sequence. Take, for each $\alpha_n$ the   eigenfunction ${\bf \Phi}_{n} \in\mathcal{H}_\Gamma^1$ of (\ref{eq:weighted_elliptic_vector}) verifying the normalization~\eqref{eq:normalization}, with $\Lambda_{n}$ standing for the principal eigenvalue. Note that the norms for the eigenfunctions are bounded in $\mathcal{H}_\Gamma^1$ due to Lemma \ref{lemma:bounds:problem_alpha}. Furthermore, since $\mathcal{H}_{\Gamma}^1$ is compactly embedded in $\mathcal{L}^2$, there exist subsequences  $\{\alpha_n\}_{n \in \mathbb{N}}$ and $\{ {\bf \Phi}_{n} \}_{n \in \mathbb{N}}$ and a function ${\Phi}=(\phi_1,\phi_2)\in \mathcal{L}^2$ such that
  $$\lim_{n \to \infty} \| {\bf \Phi}_{n} - {\Phi} \|_{\mathcal{L}^2} = 0,$$   and weakly in $\mathcal{H}_\Gamma^1$.

 First we proof that the limit function $\Phi$ vanishes outside the domains $\Omega_{0,i}$, that is
\begin{equation}
    \phi_{i} = 0 \quad \text{in} \quad \Omega \setminus \Omega_{0,i}.
    \label{eq:conditions_infinity}
\end{equation}
Indeed,  applying Hölder's inequality, we obtain
\begin{equation*}
\begin{aligned}
    \left| \int_{\Omega_i} a_i \varphi_{i,n}^2 -  \int_{\Omega_i} a_i \phi_{i}^2 \right| & \leq \int_{\Omega_i} a_i  |\varphi_{i,n} + \phi_{i} | \left| \varphi_{i,n} - \phi_{i} \right| \\ &\leq \max_{\overline{\Omega}_i} a_i \cdot \|  \varphi_{i,n} - \phi_{i} \|_{L^2(\Omega_i)}
     \|  \varphi_{i,n} + \phi_{i} \|_{L^2(\Omega_i)}\\
     &\leq C \|  \varphi_{i,n} - \phi_{i} \|_{L^2(\Omega_i)} \rightarrow 0,\quad \text{as }n \to \infty,
    \end{aligned}
\end{equation*}where $C>0$. Hence, we conclude, thanks to Lemma~\ref{lemma:integral.0}, that
\begin{equation*}
    \int_{\Omega_1} a_1 \phi_{1}^2 +  \int_{\Omega_2} a_2 \phi_{2}^2 = 0,
\end{equation*}which completes the proof of (\ref{eq:conditions_infinity}), since $a_i$ are nonnegative $\Omega_i$ and identically zero in $\Omega_{0,i}$.

The next step is to prove that $\{ {\bf \Phi}_{n} \}_{n \in \mathbb{N}}$ is actually a Cauchy sequence in $\mathcal{H}_\Gamma^1$ which allows us to conclude that $\Phi \in \mathcal{H}_\Gamma^1$ and
\begin{equation}
    \lim_{n \to \infty} \| {\bf \Phi}_{n} - {\Phi} \|_{\mathcal{H}_\Gamma^1} = 0.
    \label{eq:limit_theorem}
\end{equation} Consider $n < m$ such that $0 < \alpha_n < \alpha_m$. For ${\bf \Phi}_{n}$ and ${\bf \Phi}_{m}$ we claim that
$$D_{n,m} := \| \nabla ({\bf \Phi}_{n} - {\bf \Phi}_{m} ) \|_{\mathcal{L}^2} = \int_{\Omega_1} \left| \nabla (\varphi_{1,n} - \varphi_{1,m}) \right|^2 + \int_{\Omega_2} \left| \nabla (\varphi_{2,n} - \varphi_{2,m}) \right|^2\to 0,$$
as $n,m\to\infty$. To show the claim, we consider the terms \begin{equation*}
    D_{n,m}^1 := \int_{\Omega_1} \left| \nabla (\varphi_{1,n} - \varphi_{1,m}) \right|^2 \quad \text{and} \quad D_{n,m}^2 :=  \int_{\Omega_2} \left| \nabla (\varphi_{2,n} - \varphi_{2,m}) \right|^2.
\end{equation*} 
Since $\gamma_2\geq\gamma_1$, we get that
\begin{equation*}
\begin{aligned}
        D_{n,m}^1 \leq \frac{\gamma_2}{\gamma_1} D_{n,m}^1&=\frac{\gamma_2}{\gamma_1}\left ( \int_{\Omega_1} \left| \nabla \varphi_{1,n} \right|^2 + \int_{\Omega_1} \left| \nabla \varphi_{1,m} \right|^2 - 2  \int_{\Omega_1} \langle \nabla \varphi_{1,n} , \nabla \varphi_{1,m} \rangle\right) \\
        &= \frac{\gamma_2}{\gamma_1}\int_{\Omega_1} (\lambda_{n} m_1 \varphi_{1,n} - \alpha_n a_1 \varphi_{1,n} ) \varphi_{1,n} + \gamma_2 \int_\Gamma (\varphi_{2,n} - \varphi_{1,n}) \varphi_{1,n} \\
        &\qquad+ \frac{\gamma_2}{\gamma_1}\int_{\Omega_1} (\lambda_{m} m_1 \varphi_{1,m} - \alpha_m a_1 \varphi_{1,m} ) \varphi_{1,m} + \gamma_2 \int_\Gamma (\varphi_{2,m} - \varphi_{1,m}) \varphi_{1,m} \\
        &\qquad-2\frac{\gamma_2}{\gamma_1} \int_{\Omega_1} (\lambda_{n} m_1 \varphi_{1,n} - \alpha_n a_1 \varphi_{1,n} ) \varphi_{1,m} -2 \gamma_2 \int_\Gamma (\varphi_{2,n} - \varphi_{1,n}) \varphi_{1,m},
\end{aligned}
\end{equation*} and
\begin{equation*}
\begin{aligned}
        D_{n,m}^2 &=  \int_{\Omega_2} \left| \nabla \varphi_{2,\alpha_n} \right|^2 + \int_{\Omega_2} \left| \nabla \varphi_{2,\alpha_m} \right|^2 - 2  \int_{\Omega_2} \langle \nabla \varphi_{2,\alpha_n} , \nabla \varphi_{2,\alpha_m} \rangle \\ &=  \int_{\Omega_2} (\lambda_{\alpha_n} m_2 \varphi_{2,\alpha_n} - \alpha_n a_2 \varphi_{2,\alpha_n} ) \varphi_{2,\alpha_n} - \gamma_2 \int_\Gamma (\varphi_{2,\alpha_n} - \varphi_{1,\alpha_n}) \varphi_{2,\alpha_n} \\ &+\int_{\Omega_2} (\lambda_{\alpha_m} m_2 \varphi_{2,\alpha_m} - \alpha_m a_2 \varphi_{2,\alpha_m} ) \varphi_{2,\alpha_m} - \gamma_2 \int_\Gamma (\varphi_{2,\alpha_m} - \varphi_{1,\alpha_m}) \varphi_{2,\alpha_m} \\ &-2\int_{\Omega_2} (\lambda_{\alpha_n} m_2 \varphi_{2,\alpha_n} - \alpha_n a_2 \varphi_{2,\alpha_n} ) \varphi_{2,\alpha_m} +2 \gamma_2 \int_\Gamma (\varphi_{2,\alpha_n} - \varphi_{1,\alpha_n}) \varphi_{2,\alpha_m}.
\end{aligned}
\end{equation*}
By rearranging terms we get
\begin{equation}
\label{eq:Dnm.bounded}
\begin{aligned}
        D_{n,m} &\leq \frac{\gamma_2}{\gamma_1}\int_{\Omega_1}\left[ m_1 \lambda_{n}( \varphi_{1,n}- \varphi_{1,m})^2
        +m_1(\lambda_m-\lambda_n)\varphi^2_{1,m}\right]\\
        &\qquad+
        \int_{\Omega_2}\left[ m_2 \lambda_{n}( \varphi_{2,n}- \varphi_{2,m})^2+ m_2(\lambda_m-\lambda_n)\varphi^2_{2,m}\right] + A_{n,m} + G_{n,m},
        \end{aligned}
\end{equation}

where
\begin{equation*}\begin{aligned}
    A_{n,m} &=  -\frac{\gamma_2}{\gamma_1} \int_{\Omega_1} a_1 \left[\alpha_n (\varphi_{1,n}-\varphi_{1,m})^2+ (\alpha_m-\alpha_n)\varphi_{1,m}^2\right]\\
    &\qquad  -
   \int_{\Omega_2} a_2 \left[\alpha_n (\varphi_{2,n}-\varphi_{2,m})^2+ (\alpha_m-\alpha_n)\varphi_{2,m}^2\right],
   \end{aligned}
\end{equation*} and
\begin{equation*}
\begin{aligned}
G_{n,m} &= \gamma_2 \left(\int_\Gamma (\varphi_{2,n} - \varphi_{1,n}) \varphi_{1,n} +  (\varphi_{2,m} - \varphi_{1,m}) \varphi_{1,m}-2(\varphi_{2,n} - \varphi_{1,n}) \varphi_{1,_m}\right)
 \\ &
 \qquad  - \gamma_2\left( \int_\Gamma (\varphi_{2,n} - \varphi_{1,n}) \varphi_{2,n} + (\varphi_{2,m} - \varphi_{1,m}) \varphi_{2,m} - 2 (\varphi_{2,n} - \varphi_{1,n}) \varphi_{2,\alpha_m}\right).
\end{aligned}
\end{equation*}
Notice that, since $\{\alpha_n\}$ is an increasing sequence we have $A_{n,m} \leq 0$. Moreover,
\begin{equation*}
            G_{n,m} = -\gamma_2 \int_\Gamma (\varphi_{2,n} - \varphi_{1,n})^2+ (\varphi_{2,m} - \varphi_{1,m})^2 -2(\varphi_{2,n} - \varphi_{1, n})(\varphi_{2,m} - \varphi_{1,m}) \leq 0.
    \end{equation*}
Finally, we can pass to the limit in~\eqref{eq:Dnm.bounded}
and, since $\lambda_n\to\ell$ and ${\bf \Phi}_n$ converges in $\mathcal{L}^2$,
we get that $D_{n,m}\to 0$ and hence~\eqref{eq:limit_theorem}.
\end{proof}

To state the main result of this section, namely the characterization of the limit problem for~\eqref{eq:weighted_elliptic_vector}, we consider the pair of uncoupled equations given by
\begin{equation}
    \label{eq:alpha_m}
-\Delta \varphi_i = \lambda_{m_i} m_i(x) \varphi_i, \quad\text{in }\Omega_{0,i},
\end{equation}
under homogeneous Dirichlet boundary conditions. Recall that due to Lemma~\ref{lemma:escalar.aux}, both equations have a unique positive eigenvalue, that we denote just $\lambda_{m_i}$ or $\lambda_{m_i}^{\Omega_{0,i}}[-\Delta;\mathcal{D}]$. Also, we define
\begin{equation}
    \lambda_\infty := \min \{ \lambda_{m_1},\lambda_{m_2}\},
    \label{def:lambda_infinity}
\end{equation}
and consider the problems
\begin{equation}
    \begin{cases}
    \begin{aligned}
       -\Delta \varphi_{i,\infty} &= \lambda_\infty m_i(x) \varphi_{i,\infty}, \  &&\text{in} \ \Omega_{0,i},  \\
       \varphi_{i,\infty}& =0, \quad &&\text{on} \  \partial \Omega_{0,i}.
       \end{aligned}
    \end{cases}
      \label{eq:alpha_infinity}
\end{equation}

\begin{theorem}
     For each fixed $\alpha>0$, let $(\Lambda_\alpha, \bf{\Phi_\alpha})$ be a solution to~{\rm\eqref{eq:weighted_elliptic_vector}}. Then,  Problem~{\rm \eqref{eq:weighted_elliptic_vector}}  converges to  Problem {\rm \eqref{eq:alpha_infinity}}  when $\alpha$ goes to $\infty$ in the sense that
      $$
      \ell:=\lim_{\alpha\to\infty} \Lambda_\alpha=\lambda_\infty,\quad\text{and}
      \quad  \Phi:=\lim_{\alpha\to\infty}{\bf \Phi}_\alpha={\bf\Phi}_\infty,$$
      where ${\bf \Phi}_\infty=(\varphi_{1,\infty}, \varphi_{2,\infty})$ is the solution of~\eqref{eq:alpha_infinity}.
    \label{teo:convergence_alpha}
\end{theorem}
Let us comment several possible situations concerning the limit problem depending on the geometrical configuration of the vanishing subdomains $\Omega_{0,i}$.
It is clear that if
\begin{equation}
\label{eq:equality_lambda_inf}
\lambda_\infty =\lambda_{m_1}=\lambda_{m_2},
\end{equation}
then, (\ref{eq:alpha_infinity}) corresponds to the system~\eqref{eq:alpha_m} and we get for two  positive uncoupled eigenfunctions that concentrate on the corresponding subdomains $\Omega_{0,i}$: $$\varphi_{i,\infty} > 0 \quad \text{in} \quad \Omega_{0,i}.$$
On the other hand, assume, without loss of generality, that
\begin{equation}
\label{eq:lambda_inf_m1_m2}
\lambda_\infty =\lambda_{m_1}<\lambda_{m_2},
\end{equation}
which might happen if, for instance, $\Omega_{0,2}$ is smaller than  $\Omega_{0,1}$ and $m_1(x) = m_2(x)$ for every $x \in \Omega$.
In this situation, $\lambda_\infty$ is not the principal eigenvalue for $$
    -\Delta \varphi_2 = \lambda_{m_2} m_2(x) \varphi_2, \quad\text{in }\Omega_{0,2},
  $$
and, since $\lambda_\infty<\lambda_{m_2}$, necessarily  $\varphi_{2,\infty}\equiv 0$. Therefore, the limit eigenfunction concentrates on $\Omega_{0,1}$, being zero in the rest of the domain $\Omega$. Hence, we have
\begin{equation}
\label{problem:alpha_infinity_m}
{\bf \Phi}_\infty=(\varphi_{1,\infty} , \varphi_{2,\infty}): =
    \begin{cases}
    \begin{aligned}
    &(\varphi_{1,\infty} , \varphi_{2,\infty}) \quad &\text{if} \quad \lambda_{m_1}=\lambda_{m_2}, \\
         &(\varphi_{1,\infty} , 0) \quad &\text{if} \quad \lambda_{m_1}<\lambda_{m_2}, \\
         &(0,\varphi_{2,\infty}) \quad &\text{if} \quad \lambda_{m_2}<\lambda_{m_1}.
   \end{aligned}
    \end{cases}
\end{equation}

\begin{proof}[Proof of Theorem~{\rm\ref{teo:convergence_alpha}}]
First we show, as a  consequence of Lemma~\ref{lemma:convergence.phi},  that the limit function $\Phi$ is such that
 \begin{equation}\label{eq:Phi_inft.H}{\Phi}\in H_{0}^{1}(\Omega_{0,1}) \times H_{0}^{1}(\Omega_{0,2}).\end{equation}
Subsequently, for a sufficiently small $\delta >0$ consider  the open sets
\begin{equation}
    \Omega_{\delta,i} := \{ x \in \Omega_i : \text{dist}(x,\Omega_{0,i}) < \delta \}.
    \label{def:omega_delta}
\end{equation}
so that, according to (\ref{eq:conditions_infinity}), ${\Phi} \in H_{0}^{1}(\Omega_{\delta,1}) \times H_{0}^{1}(\Omega_{\delta,2})$ and hence there exists $\delta_0 > 0$ such that
\begin{equation*}
    {\Phi} \in \bigcap_{0<\delta < \delta_0} (H_{0}^{1}(\Omega_{\delta,1}) \times H_{0}^{1}(\Omega_{\delta,2})).
\end{equation*}On the other hand, since the sets $\Omega_{0,i}$ are smooth subdomains of $\Omega$, they are stable in the sense of Babuska and Výborný, see \cite{Babuska1965} for further details, and, therefore,
\begin{equation*}
    H_{0}^{1}(\Omega_{0,1}) \times H_{0}^{1}(\Omega_{0,2}) = \bigcap_{0<\delta < \delta_0} (H_{0}^{1}(\Omega_{\delta,1}) \times H_{0}^{1}(\Omega_{\delta,2})).
\end{equation*}
Hence~\eqref{eq:Phi_inft.H} holds.
Finally, passing to the limit in the weak formulation of (\ref{eq:weighted_elliptic_vector}), we show that ${\Phi}$ is a weak solution of the problem (\ref{eq:alpha_infinity}). For this, we consider the test function $${\bf{\vartheta}} = (v_1,v_2)  \in C_0^\infty (\Omega_{0,1}) \times C_0^\infty (\Omega_{0,2}).$$ Thus, multiplying (\ref{eq:weighted_elliptic_vector}) by ${\bf \vartheta}$, assuming $\alpha = \alpha_n$ with $n>1$ and integrating by parts, we get for the left hand side:
\begin{equation*}
     \int_{\Omega_i} (-\Delta + \alpha_n a_i) \varphi_{i,n} v_i =    \int_{\Omega_{0,i}} -\Delta \varphi_{i,n} v_i  = \int_{\Omega_{0,i}} \nabla \varphi_{i,n}  \nabla v_i - \int_{\partial\Omega_{0,i}} \frac{\partial  \varphi_{i,n}}{\partial \boldsymbol{\nu_i}} v_i
         =  \int_{\Omega_{0,i}} \nabla \varphi_{i,n}  \nabla v_i.
   \end{equation*}
Hence
$$
\int_{\Omega_{0,i}} \nabla \varphi_{i,n}  \nabla v_i= \Lambda_{n} \int_{\Omega_{0,i}} v_i m_i\varphi_{i,n} $$
and passing to the limit as $n \to \infty$, it is apparent that $\Phi$ provides a weak solution of the uncoupled problem
\begin{equation*}
    - \Delta \varphi_{i} = \ell m_i(x) \varphi_{i,} \quad \text{in} \quad \Omega_{0,i},
\end{equation*}  where $\ell = \lim_{n \to \infty} \Lambda_{n}$, see Lemma~\ref{lemma.limit.lambda_alpha}. Hence $\Phi={\bf \Phi}_\infty$ and by uniqueness of the first eigenvalue,  $\ell=\lambda_\infty$.  Furthermore, by elliptic regularity ${\bf \Phi}_\infty$ is indeed a classical solution,  ${\bf \Phi}_\infty\in C^{2,\eta}(\overline{\Omega}_{0,1}) \times C^{2,\eta}(\overline{\Omega}_{0,2})$.
\end{proof}

\section{Existence of solutions}
\label{sect.existence}
In this section, we will characterise the existence and uniqueness of positive solutions ${\bf u}$ of problem (\ref{eq:original_problem})--(\ref{eq:conditions}) in terms of the parameter $\lambda$. First, following Definition~\ref{def.subsuper}, we define sub and supersolutions to problem~\eqref{eq:original_problem}--\eqref{eq:conditions}.

\begin{definition}
    We say $\underline{\bf u} \geq 0$, $u_i\in C^2(\Omega_i)\cap C^1(\overline{\Omega}_i)$ is a nonnegative subsolution (respectively  $\overline{\bf u} \geq 0$ is a nonnegative supersolution) to {\rm \eqref{eq:original_problem}}--{\rm \eqref{eq:conditions}} if
    \begin{equation*}
        ( -\Delta - \lambda {\bf m} + {\bf a\underline{u}}^{p-1})\underline{\bf u} \leq 0 \quad \quad \text{in} \quad \Omega, \quad \text{and} \quad \mathcal{I}(\underline{\bf u}) \preceq 0 \quad \text{ on } \ \Gamma\cup \partial \Omega,\quad \text{(resp. $\geq$ and $\succeq$)},
    \end{equation*}
\end{definition}

\begin{theorem}
\label{theorem:solution_degenerate}
   Problem {\rm \eqref{eq:original_problem}--\eqref{eq:conditions}} admits a unique positive solution ${\bf u} \in \mathcal{C}^{2,\eta}_\Gamma $ if and only if $\lambda$ is such that
    \begin{equation}
         0 < \lambda_* < \lambda < \lambda_\infty,
        \label{eq:lambda_condition_nondegenerate}
    \end{equation}
where $\lambda_*:=\Lambda_m^\Omega[-\Delta;\mathcal{I}]$ and $\lambda_\infty$ is given by~\eqref{def:lambda_infinity}.
\end{theorem}

\begin{remark}
\label{remark:solution_degenerate}
    If ${\bf a}$ is positive everywhere, then the existence of a positive solution is guaranteed for all $\lambda>\lambda_*$.
\end{remark}

\begin{remark}
\label{rem:sigma=0}
Observe that,  by definition,  $$\Sigma_1^\Omega [- \Delta - \lambda_* {\bf m};\mathcal{I}]=0\quad \text{and}\quad \sigma_1^{\Omega_{0,1}} [ -\Delta - \lambda_\infty m_1;\mathcal{D}]=0.$$
\end{remark}

\begin{proof}
   Let $\lambda_\infty=\lambda_{m_1}\leq \lambda_{m_2}$. The other case can be handled in a similar way.

   Assume that there exists a positive solution ${\bf u}$ of~\eqref{eq:original_problem}--\eqref{eq:conditions}. Then
   $$
   (-\Delta-\lambda{\bf m}+{\bf a}{\bf u}^{p-1}){\bf u}=0,
   $$ and thanks to the uniqueness of the principal eigenvalue, we have that
     $$
   \Sigma_1^\Omega [- \Delta - \lambda {\bf m} + {\bf a u}^{p-1};\mathcal{I}] = 0.
   $$
 Hence, due to the monotonicity of the principal eigenvalue with respect to the potential,
   $$
    \Sigma_1^\Omega [- \Delta - \lambda {\bf m};\mathcal{I}] < \Sigma_1^\Omega [- \Delta - \lambda {\bf m} + {\bf a u}^{p-1};\mathcal{I}] = \Sigma_1^\Omega [- \Delta - \lambda_* {\bf m};\mathcal{I}],
   $$
 and therefore $\lambda_* < \lambda $.
Moreover, arguing analogously as in the proof of Lemma~\ref{lemma.limit.lambda_alpha},  since the principal eigenvalue is monotone with respect to the domain, it follows that $$
\begin{aligned}\sigma_1^{\Omega_{0,1}} [ -\Delta - \lambda_\infty m_1;\mathcal{D}] &= \Sigma_1^\Omega [-\Delta - \lambda {\bf m} + {\bf a u}^{p-1};\mathcal{I}]\\ &< \min\{\sigma_1^{\Omega_{0,1}} [ -\Delta - \lambda m_1;\mathcal{D}], \sigma_1^{\Omega_{0,2}} [ -\Delta - \lambda m_2;\mathcal{D}]\},
\end{aligned}$$
   from where $\lambda < \lambda_\infty$ follows.

On the other hand, fix $\lambda=\bar \lambda$ so  that \eqref{eq:lambda_condition_nondegenerate} holds. To obtain the existence of positive solutions for problem \eqref{eq:original_problem}--\eqref{eq:conditions}
   we apply the method of sub and supersolutions, see~\cite{AC-LG2} and~\cite{WangSu}.

Thus, for $\bar\lambda$, let ${\bf \Phi}_0$ be the principal eigenfunction associated with the principal eigenvalue {$\Sigma_1^\Omega [ -\Delta - \bar\lambda {\bf m}; \mathcal{I}]$}. By assumption, recall also Remark~\ref{rem:sigma=0},
$$\Sigma_1^\Omega [- \Delta - \bar\lambda {\bf m};\mathcal{I}] <\Sigma_1^\Omega [- \Delta - \lambda_* {\bf m};\mathcal{I}]=0.$$ Take $0 < \epsilon \ll 1$, and consider $\epsilon{\bf \Phi}_0$. Then
  \begin{equation*}
       \epsilon (-\Delta - \bar\lambda {\bf m}) {\bf \Phi_{0}} = \epsilon \Sigma_1^\Omega [- \Delta - \bar \lambda {\bf m};\mathcal{I}] {\bf \Phi_{0}} < -{\bf a} \epsilon^p {\bf \Phi}_{0}^p, \quad \text{in} \quad \Omega.
   \end{equation*}
and $\mathcal{I}(\epsilon{\bf \Phi_0})=0$, so that $\epsilon {\bf \Phi}_0$ is a subsolution to \eqref{eq:original_problem}--\eqref{eq:conditions}.

     To define a supersolution let us consider for sufficiently small $\delta > 0$ the sets $\Omega_{\delta,i}$ defined in \eqref{def:omega_delta} and let $\varphi_{\delta,i}$, denote the principal eigenfunction of
       \begin{equation}
       \label{eq:aux.varphi.delta}
    \begin{cases}
    \begin{aligned}
       ( -\Delta  - \bar\lambda m_i(x)) \varphi_{\delta,i} &= \sigma \varphi_{\delta,i},        &&\text{in} \ \Omega_{\delta,i},\\
        \varphi_{\delta,i} &=0,  &&\text{on} \  \partial \Omega_{\delta,i}.
    \end{aligned}
    \end{cases}
      \end{equation} with eigenvalue  $\sigma_1^{\Omega_{\delta,i}}[-\Delta - \bar\lambda m_i;\mathcal{D}]$. By continuous dependence of the eigenvalues with respect to the domains, see Proposition~\ref{prop:eigvals} and~\cite{CanoLopez2002}
     $$
     \lim_{\delta \to 0} \sigma_1^{\Omega_{\delta,i}} [-\Delta - \bar\lambda m_i;\mathcal{D}] = \sigma_1^{\Omega_{0,i}} [-\Delta - \bar\lambda m_i;\mathcal{D}].$$
  Hence, due to the assumption $\sigma_1^{\Omega_{0,i}} [-\Delta - \bar\lambda m_i;\mathcal{D}]>\sigma_1^{\Omega_{0,1}} [-\Delta - \lambda_\infty m_1;\mathcal{D}]=0$ ,
    for sufficiently small $\delta > 0$, we have
    \begin{equation}
       0<\sigma_1^{\Omega_{\delta,i}} [-\Delta - \bar\lambda m_i;\mathcal{D}] < \sigma_1^{\Omega_{0,i}} [-\Delta - \bar\lambda m_i;\mathcal{D}].
        \label{cond:princ_eigv_delta}
    \end{equation}

Now, define ${\bf \Psi} $ as
\begin{equation}
\label{def.psi.soper}
    \psi_i := \begin{cases}
        \begin{aligned}
            & \varphi_{\delta,i} \quad &&\text{in} \ \overline{\Omega}_{\delta/2,i}, \\
            & \varphi_{+,i} \quad &&\text{in} \ \Omega_i \setminus \overline{\Omega}_{\delta/2,i},
        \end{aligned}
    \end{cases}
\end{equation}where $\varphi_{+,i}$ is any smooth extension, positive and separated away from zero, chosen such that ${\bf \Psi} \in {\mathcal C}^{2,\eta}_\Gamma$. Take $M$, sufficiently large $M$, and consider $M {\bf \Psi}$. It is clear that, by construction,  $\mathcal{I}(M {\bf \Psi})=0$. Moreover,  since $\varphi_{\delta,i}>0$ in $\overline{\Omega}_{\delta/2,i}$ and $M>0$, from \eqref{cond:princ_eigv_delta}, for $x\in \Omega_{\delta/2,i}$ we have that
\begin{equation*}
\begin{aligned}
(-\Delta - \bar\lambda m_i)M \varphi_{\delta,i} + a_i M^p \varphi_{\delta,i}^p \geq (-\Delta - \bar\lambda m_i)M \varphi_{\delta,i} =  \sigma_1^{\Omega_{\delta,i}}[-\Delta - \bar\lambda m_i;\mathcal{D}] M \varphi_{\delta,i} \geq 0.
\end{aligned}
\end{equation*}
Also, since $a_i$ and $\varphi_{+,i}$ are positive and bounded away from zero in $ \Omega_i \setminus \overline{\Omega}_{\delta/2,i}$ it follows
\begin{equation*}
    (-\Delta - \bar \lambda m_i)M \varphi_{+,i} + a_i M^p \varphi_{+,i}^p \geq 0,
\end{equation*}for sufficiently large  $M >1$.
 Hence, $M {\bm \Psi}$ provides a supersolution of \eqref{eq:original_problem}--\eqref{eq:conditions}.

Uniqueness can be proved using the standard arguments shown in~\cite{BrezisOswald1986}, see also~\cite{Suarezetal} for an interface problem.
\end{proof}

Next we prove the existence of a branch of positive solutions of \eqref{eq:original_problem}--\eqref{eq:conditions}  emanating from the trivial solution $({\bf u} , \lambda) = (0, \lambda)$ at $\lambda = \lambda_*.$ This, indeed, means that $\lambda_*$ is a bifurcation point to a smooth curve of solutions of \eqref{eq:original_problem}--\eqref{eq:conditions}, in the sense of~\cite{CR},~\cite{CRs}. Moreover, we will show that the solutions are monotone with respect to the parameter $\lambda$ and the branch of positive solutions goes to infinity.

\begin{theorem}
\label{theorem:bifurcation}
    For any fixed $\lambda > \lambda_*$, let ${\bf u}_\lambda\in \mathcal{C}^{2,\eta}_\Gamma$ be the unique positive solution of Problem {\rm\eqref{eq:original_problem}--\eqref{eq:conditions}}. Then $\lambda_*$ is the unique bifurcation point for \eqref{eq:original_problem}--\eqref{eq:conditions} to  a branch of positive solutions $({\bf u}_\lambda,\lambda)$ emanating from the trivial solution ${\bf u}\equiv 0$.
    Moreover, the map $\lambda\mapsto {\bf u}_\lambda$ is $C^1$ and increasing.
\end{theorem}

\begin{proof}
Consider the operator $\mathfrak{F}: {\mathcal{C}_\Gamma  \times \mathbb{R} \rightarrow \mathcal{C}_\Gamma},$ defined by
 $$
 \mathfrak{F}({\bf u},\lambda) := ({\rm {\bf I}} -  (-\Delta)^{-1} (\lambda{\bf m} - {\bf a u}^{p-1})){\bf u},
 $$
 where $(-\Delta)^{-1}$ denotes the inverse operator of $(-\Delta)$. It is clear that for each $\lambda>\lambda_*$, since ${\bf u}_\lambda \in \mathcal{C}_\Gamma$ is the positive solution to \eqref{eq:original_problem}--\eqref{eq:conditions}, then $\mathfrak{F}({\bf u}_\lambda, \lambda)=0$. Besides, we know that $\mathfrak{F}$ is of class $C^1$ and, by elliptic regularity, $\mathfrak{F}(\cdot, \lambda)$ is a compact perturbation of the identity for every $\lambda \in \mathbb{R}$.
Moreover,  $\mathfrak{F}(0, \lambda)=0$ for all $\lambda \in \mathbb{R}$.

We first prove that $\lambda_*$ is a bifurcation point. Differentiating $\mathfrak{F}$ with respect to ${\bf u}$ we get
$$
D_{\bf u} \mathfrak{F}({\bf u}, \lambda)={\rm {\bf I}} - (-\Delta)^{-1} (\lambda {\bf m} -  p {\bf a} {\bf u}^{p-1}), \quad \text{and} \quad D_{\bf u} \mathfrak{F}({\bf u}, \lambda)\Big|_{\mathbf{u}=0}=({\bf I} - \lambda (-\Delta)^{-1} {\bf m} )
$$
Note that, being a compact perturbation of the identity map, see~\cite{Brezis}, the linear operator
$D_{\bf u} \mathfrak{F}(0, \lambda)$ is a Fredholm operator of index zero. Moreover, we denote
$$
\mathfrak{D}_*:=D_{\bf u} \mathfrak{F}(0, \lambda_*)=({\bf I} - \lambda_* (-\Delta)^{-1} {\bf m} )
$$
and take a function ${\bf \Phi}\in\mathcal{C}_\Gamma^\eta$ with ${\bf \Phi}\neq 0$. First observe that  ${\bf \Phi}\in{\rm Ker}({\mathfrak{D}_*})$, if and only if ${\bf \Phi}$ is an eigenfunction of~\eqref{eq:problem_lambda*} (with ${\bf c}=0$) with eigenvalue $\lambda_*$. Then, as mentioned before, Theorem~\ref{thm:unique.lambda.eignefunction}, since~\eqref{eq:problem_lambda*}  has a unique positive solution, we have
$$
{\rm Ker}({\mathfrak{D}_*})={\rm span}[{\bf \Phi}].$$
On the other hand, let
$$
\mathfrak{D}:=D_\lambda (D_{\bf u} \mathfrak{F}(0, \lambda))\Big|_{\lambda=\lambda_*}=- (-\Delta)^{-1}{\bf m}
$$
We claim that the following transversality condition (check \cite{CR}) holds:
\begin{equation}
\label{cond:crandall_rab_proof_bir}
    \mathfrak{D} {{\bf \Phi}} \notin {\rm Range}(\mathfrak{D}_*).
\end{equation}
Indeed, suppose by contradiction that there exists ${\bf v} \in \mathcal{C}_\Gamma$ such that $\mathfrak{D}_*{\bf v} = \mathfrak{D} {\bf \Phi}$, i.e.,
$$
({\bf I} - \lambda_* (-\Delta)^{-1} {\bf m}){\bf v} = -(-\Delta)^{-1} {\bf m}{\bf \Phi},$$
which is equivalent to
\begin{equation}
    \label{eq:contradiction.transversality}
    ((-\Delta) - \lambda_* {\bf m}) {\bf v} = -{\bf m }{\bf \Phi}.
\end{equation}
Hence, thanks to elliptic regularity, see~\cite{WangSu}, we have that ${\bf v} \in \mathcal{C}_\Gamma^{2,\eta}$.
 Now, multiply~\eqref{eq:contradiction.transversality} by $(\gamma_2\varphi_1,\gamma_1\varphi_2)$ and integrating by parts it follows that
\begin{equation}
    \label{eq:crandall-rab.contradiction}
\begin{aligned}
        &\gamma_2\int_{\Omega_1}\left( \nabla v_1\nabla\varphi_1 - \lambda_* m_1(x)  v_1\varphi_1 \right)+\gamma_1\int_{\Omega_2}\left( \nabla v_2\nabla\varphi_2 - \lambda_* m_2(x)v_2 \varphi_2\right)\\
        &\qquad\quad+ \int_{\Gamma} (-\gamma_2 \frac{\partial v_{1}}{\partial \boldsymbol{\nu_1}}\varphi_1 +\gamma_1\frac{\partial v_2}{\partial \boldsymbol{\nu_1}}\varphi_2) = - \left( \int_{\Omega_1} \gamma_2 m_1(x) \varphi_{1}^2 + \int_{\Omega_2} \gamma_1 m_2(x) \varphi_{2}^2 \right)<0
\end{aligned}
\end{equation}
On the other hand, recall that ${\bf\Phi}$  is an eigenfunction of~\eqref{eq:problem_lambda*} (with ${\bf c}=0$) with eigenvalue $\lambda_*$. Hence, we multiply each equation in~\eqref{eq:problem_lambda*} by ${\bf v}=(v_1,v_2)$ to obtain
$$
\begin{aligned}
        &\int_{\Omega_1}\left( \nabla v_1\nabla\varphi_1 - \lambda_* m_1(x)  v_1\varphi_1 \right)
       -\int_{\Gamma} \frac{\partial \varphi_{1}}{\partial \boldsymbol{\nu_1}}v_1=0
\end{aligned}
$$
and
$$
\begin{aligned} \int_{\Omega_2}\left( \nabla v_2\nabla\varphi_2 - \lambda_* m_2(x)v_2 \varphi_2\right) +\int_{\Gamma}\frac{\partial \varphi_2}{\partial \boldsymbol{\nu_1}}v_2 = 0.
\end{aligned}
$$
We can replace these two equations in~\eqref{eq:crandall-rab.contradiction} and get
$$
\gamma_1\gamma_2(\varphi_2-\varphi_1)(u_2-u_1)+\gamma_1\gamma_2(\varphi_2-\varphi_1)(u_1-u_2)<0.
$$
Hence, there is a contradiction and \eqref{cond:crandall_rab_proof_bir} holds.

Therefore, applying Crandall-Rabinowitz Theorem, see~\cite{CR}, we conclude that $\lambda_*$ is a bifurcation point from $(0,\lambda)$ to a branch of positive solutions $({\bf u}_\lambda,\lambda)$ of \eqref{eq:original_problem}--\eqref{eq:conditions}.

Next, to see the regularity of the branch of solutions, let us consider the differential operator
$$
D_{\bf u} \mathfrak{F}({\bf u}, \lambda)\Big|_{\mathbf{u}=\mathbf{u}_\lambda}= \left({\bf I} - (-\Delta) (\lambda {\bf m} -  p {\bf a} {\bf u}^{p-1}_\lambda) \right).$$
Observe that, as before, being a compact perturbation of the identity map, see~\cite{Brezis}, the linear operator
 is a Fredholm operator of index zero.
Furthermore, we claim that it is injective, and therefore a linear topological isomorphism. Indeed, consider ${\bf u} \in \mathcal{C}_\Gamma$ such that
$$
D_{\bf u} \mathfrak{F}({\bf u_\lambda}, \lambda){\bf u}=0.
$$
Then, we have that
\begin{equation}
\label{eq:problem_proof_bif}
    (-\Delta - \lambda{\bf m} + p{\bf a} {\bf u}^{p-1}_\lambda){\bf u} = 0,
\end{equation}
On the other hand, since ${\bf u}_\lambda$ is a solution to~\eqref{eq:original_problem}--\eqref{eq:conditions}, we have that
\begin{equation*}
    \label{eq:cond_eig_proof_bif}
        \Sigma_1^{\Omega}[ -\Delta - \lambda{\bf m} + {\bf a} {\bf u}^{p-1}_\lambda;\mathcal{I}]=0.
    \end{equation*}
Hence,  thanks to the monotonicity of the principal eigenvalue with respect to the potential, and since $p>1$, we find that
\begin{equation}
\label{cond:bif_proof_positive_eig}
    \Sigma_1^{\Omega} [-\Delta - \lambda{\bf m} + p{\bf a} {\bf u}^{p-1}_\lambda;\mathcal{I}] > 0.
\end{equation}
This, together with \eqref{eq:problem_proof_bif}, yields to ${\bf u}=0$. Therefore,  $D_{\bf u} \mathfrak{F}(\mathbf{u}_\lambda, \lambda)$ is injective,  and hence, it is invertible.
 In addition, due to the uniqueness of positive solutions of {\rm\eqref{eq:original_problem}--\eqref{eq:conditions}} and the application of the Implicit Function Theorem it follows that the map $\lambda\mapsto {\bf u}_\lambda$ is of class $C^1$.

Finally, to see that the positive branch of solutions is monotone increasing we differentiate $\mathfrak{F}({\bf u_\lambda}, \lambda)=0$  with respect to the parameter $\lambda$ and get
$$D_{\bf u} \mathfrak{F}({\bf u_\lambda}, \lambda) \frac{d {\bf u_\lambda}}{d \lambda} + D_\lambda \mathfrak{F}({\bf u_\lambda}, \lambda)=0.$$
Then, applying the operator $(-\Delta)$ on both sides of the previous expression we get
$$ ((-\Delta) - \lambda{\bf m} + p{\bf a} {\bf u}^{p-1}_\lambda) \frac{d {\bf u}_\lambda}{d \lambda}  = {\bf m u}_\lambda.$$
Hence,  due to~\eqref{cond:bif_proof_positive_eig}, we have that
$$  \frac{d {\bf u_\lambda}}{d \lambda}  =((-\Delta) - \lambda{\bf m} + p{\bf m} {\bf u}^{p-1}_\lambda)^{-1} {\bf m u_\lambda}>0,$$
and consequently, ${\bf u_\lambda}$ is increasing with respect to $\lambda$.
\end{proof}

\section{Asymptotic behaviour of the solution}
\label{sect.asymptotic}

 In this section we analyse the asymptotic behaviour of the solution ${\bf u}_\lambda$ when the parameter $\lambda$ is in the interval $(\lambda_* , \lambda_\infty)$ but approximates to $\lambda_\infty$. Recall that $\lambda_*=\Lambda_m^\Omega[-\Delta;\mathcal{I}]$, see Theorem~\ref{theorem:solution_degenerate} and that $\lambda_{m_i}:=\lambda_{m_i}^{\Omega_{0,i}}[-\Delta;\mathcal{D}]$, see~\eqref{eq:alpha_m}.
 We will assume that $\lambda_\infty:=\lambda_{m_1}<\lambda_{m_2}$, being the analysis of the case $\lambda_{m_2}<\lambda_{m_1}$ similar. In particular we will show that  positive solutions have non-simultaneous blow-up when $\lambda$ approaches $\lambda_\infty$, in the sense that the  component $u_{\lambda,1}$  tends to infinity inside $\overline{\Omega}_{0,1}$, while $u_{\lambda,2}$ is bounded in the whole $\Omega_2$, where that compoment is defined.  Observe that when $\lambda_\infty=\lambda_{m_1}=\lambda_{m_2}$ both components blow-up.
First of all, we prove a priori bounds for the solution ${\bf u}_\lambda$ outside the refuges, independently of the value of $\lambda_\infty$. Let us define, for $\epsilon >0$ the sets
\begin{equation}
\label{def.omega.eps}
    \Omega_{\epsilon,i} := \{ x \in \Omega_i \setminus \overline{\Omega}_{0,i}: \ \text{dist}(x, \overline{\Omega}_{0,i}) > \epsilon \}.
 \end{equation}
    \begin{lemma}
\label{lemma.apriori}
    For any fixed $\lambda > \lambda_*$, let ${\bf u_\lambda}$ be the unique positive solution of problem
    \begin{equation}
    \label{eq:epsilon_prob}
\left\{\begin{array}
  {ll}
  - \Delta u_i = \lambda m_{i}(x)u_{i} - a_{i}(x)u_{i}^{p}, \quad &x\in\Omega_{\epsilon,i},\\
  u_{i}=P_i, \quad & x\in \partial\Omega_{\epsilon,i},\\
   \mathcal{I}(\mathbf{u})=0,&x\in\Gamma\cup \partial \Omega,
\end{array}\right.
\end{equation}
with $P_i>0$.
Then
     $$
\|u_{i,\lambda}\|_{C(\Omega_{\epsilon,i})}\leq K,
   $$ where $K$ is a positive constant.
\end{lemma}
\begin{proof}
    Since $\overline\Omega_{\epsilon,i}\subset (\Omega_{i}\setminus \overline \Omega_{0,i})$ we find that $\tilde{a} =\min_{x\in \Omega} \{a_i(x),\; i=1,2\}>0$. Then, there exists a constant $K>0$ such that
    $$0\geq \lambda m_{i}-\tilde{a} K^{p-1}, \quad \hbox{in}\quad \Omega_{\epsilon,i},$$
    and $K\geq P_i$ on $\partial\Omega_{\epsilon,i}$, with $\mathcal{I}(\mathbf{u})=0$ on $\Gamma\cup \partial \Omega$. Consequently, if $K$ is big enough, $(K,K)$ is a supersolution to~\eqref{eq:epsilon_prob}, concluding the proof.
\end{proof}

\begin{theorem}
    \label{theorem:blow_up_degenerate_2}
    Let $\lambda_\infty:=\lambda_{m_1}<\lambda_{m_2}$. For any fixed $\lambda\in (\lambda_*,\lambda_\infty)$, let ${\bf u_\lambda}$ be the unique positive solution of \eqref{eq:original_problem}--\eqref{eq:conditions}. Then
     $$
    u_{1,\lambda}\to\infty\  \text{uniformly on $\overline{\Omega}_{0,1}$},
   $$ as $\lambda \to \lambda_\infty$.
     \end{theorem}

\begin{proof}
    Consider a sequence $\{\lambda_n\}$ which converges from below to $\lambda_\infty$ as $n \to \infty$, and the corresponding unique solutions ${\bf u}_{\lambda_n}$, that we denote, for simplicity,  ${\bf u}_n=(u_{1,n},u_{2,n})$.  For a fixed $\lambda_n$, let ${\bf \Phi}_n=(\varphi_{1,n},\varphi_{2,n})$ be the principal eigenfunction associated with the principal eigenvalue $\lambda_n$ of \eqref{eq:weighted_elliptic_vector}  and normalised so that $\|{\bf \Phi}_n\|_{\mathcal{L}^\infty}\leq 1$. Moreover, due to the convergence of the linear problem~\eqref{eq:weighted_elliptic_vector}, see Lemma~\ref{lemma:convergence.phi} and Theorem \ref{teo:convergence_alpha}, it follows $${\bf \Phi}_n \to { \bf \Phi_\infty} \quad \text{in} \quad H^1(K_1) \times H^1(K_2),$$
    as $n \to \infty$, where ${\bf \Phi}_\infty $ is a solution of \eqref{eq:alpha_infinity} and  $K_i \subset \Omega_{0,i}$ are two compact sets.

Now, let  ${\bf v}=(v_1,v_2)=\alpha_n^{\frac{1}{p-1}}{\bf\Phi}_{n}$. It is clear that $\mathcal{I}({\bf v})=0$ is satisfied. Moreover, in $\Omega_1$, for the first component $v_1$, we have
    \begin{equation*}
    -\Delta v_1 - \lambda_n m_1 v_1+ a_1 v_1^p = - a_1 \alpha_n^{\frac{p}{p-1}} \varphi_{1,n} \left( 1- \varphi_{1,n}^{p-1} \right) \leq 0.
    \end{equation*}
Therefore, by comparison $u_{1,n}\geq v_1=\alpha_n^{\frac{1}{p-1}}\varphi_{1,n}$. Hence, since $\varphi_{1,n}$ is positive and bounded and $\alpha_n\to\infty$, we have that $u_{1,n}\to\infty$ uniformly in $K_1$, see Theorem~\ref{teo:convergence_alpha}.

Next we proof convergence up to the boundary of $\Omega_{0,1}$ for the component $u_{1,n}$. To do so, we argue by contradiction following the argument of~\cite{DuHuang1999}.

    Since $\lambda_n >0$ for all $n \geq 1$ we have that $$-\Delta u_{1,n} =  \lambda_n m_1 u_{1,n}\geq 0 \quad \text{in} \ \Omega_{0,1}.$$Then, due to the maximum principle, the minimum of $u_{1,n}$ is achieved at the boundary. Therefore it is enough to prove that
    $$u_{1,n}(x_n):= \min_{\partial \Omega_{0,1}} u_{1,n}(x) \to \infty$$as $n \to \infty$, where $x_n \equiv x_{\lambda_n} \in \partial \Omega_{0,1}$. We argue by contradiction assuming that there exists a sequence $\{x_n\}$ and a positive constant $k$ such that
    \begin{equation}
        \label{sequence_contradiction_theorem}
        u_{1,n}(x_n) \leq k \quad \text{for all} \ n \geq 1 \ \text{and} \ x_n \in \partial \Omega_{0,1}.
    \end{equation}
   Now, due to the smoothness of $\partial \Omega_{0,1}$, there exists a map $y: \partial \Omega_{0,1} \rightarrow \Omega_{0,1}$ and an integer $R>0$  such that for every $z \in \partial \Omega_{0,1}$
    \begin{equation}
    \label{ball_proof}
        B(y(z),R) \subset \overline{\Omega}_{0,1}, \quad \overline{B}(y(z),R) \cap \partial \Omega_{0,1} = \{ z \}.
    \end{equation}
    Note that the map $y$ provides us the centre of the balls in $ \Omega_{0,1}$ which satisfy \eqref{ball_proof}. Observe also, that $\partial \Omega_{0,1} \subset \Omega_1$, and the boundaries do not touch to each other, i.e., $\Gamma \cap \partial \Omega_{0,1} = \emptyset$. Moreover,  $$u_{1,n}(x) \geq u_{1,n}(x_n)$$ for each $x \in \overline{B}(y(x_n),R)$.
    Now, let us define $A(y(x_n),R) = B(y(x_n),R) \setminus \overline{B}(y(x_n),R/2)$ and consider the problem
    \begin{equation}
    \label{problem:barrier}
        \begin{cases}
            \begin{aligned}
                &-\Delta u = \lambda_n m_1 u \quad &&\text{in} \ A(y(x_n),R) , \\
                & u = u_{1,n}(x_n) + c_n \left( e^{-\delta R^2 / 4} - e^{-\delta R^2} \right) \quad && \text{on} \ \partial B(y(x_n),R/2), \\
                & u = u_{1,n}(x_n) \quad && \text{on} \ \partial B(y(x_n),R),
            \end{aligned}
        \end{cases}
    \end{equation}where
    $$
    c_n = \frac{\min_{\overline{B}(y(x_n),R/2)} u_{1,n}(x) - u_{1,n}(x_n)}{e^{-\delta R^2 / 4} - e^{-\delta R^2}}.
    $$
    It is clear that this choice of $c_n$ implies that $u_{1,n}(x_n) + c_n \left( e^{-\delta R^2 / 4} - e^{-\delta R^2} \right) \leq u_{1,n}(x)$, for all $x \in \overline{B}(y(x_n),R/2)$. Hence, $u_{1,n}$ is a supersolution of the problem \eqref{problem:barrier}.
    On the other hand, define for $\delta > 0$ and $x \in  A(y(x_n),R)$ the barrier function of exponential type
    $$
    w_n(x) := e^{-\delta \left| x - y(x_n) \right|^2} - e^{-\delta R^2}.
    $$
    It is clear that $z_n(x)=u_n(x_n) + c_n w_n(x)$ is a subsolution of \eqref{problem:barrier}. Indeed, $w_n$ can be rewritten as
    \begin{equation*}
            w_n(x) = e^{-\delta \sum_{i=1}^{N} (x_i - y_i(x_n))^2}- e^{-\delta R^2},
    \end{equation*} and consequently
    \begin{equation*}
        \frac{\partial w_n}{\partial x_i}= -2 \delta e^{-\delta \left| x - y(x_n) \right|^2} (x_i - y_i(x_n))
    \end{equation*}and
    \begin{equation*}
        \frac{\partial^2 w_n}{\partial x_i^2} = \left( -2 \delta   + 4 \delta^2  ( x_i - y_i(x_n) )^2 \right) e^{-\delta \left| x - y(x_n) \right|^2}.
    \end{equation*}Therefore, it is easy to see that we get
    \begin{equation*}
        ( - \Delta - \lambda_n m_1) w_n(x) = \left( 2 \delta N- 4 \delta^2 \left| x - y(x_n) \right|^2 - \lambda_n m_1(x_n) \right) e^{-\delta \left| x - y(x_n) \right|^2} + \lambda m_1 e^{-\delta R^2}.
    \end{equation*}Thus, for $\eta > 0$, there exists $\delta>0$ large enough such that
    $$
    (-\Delta - \lambda m_1) w_n(x) \leq - \eta < 0 \quad \text{in} \quad A(y(x_n),R),
    $$ and hence $(-\Delta - \lambda m_1)z_n(x)<0$. Therefore, due to the comparison principle for \eqref{problem:barrier} we can conclude
    \begin{equation}
        u_{1,n}(x) > u_{1,n}(x_n) + c_n w_n(x) \quad \text{for every} \quad x \in \overline{A}(y(x_n),R).
        \label{inequality_proof_blow_up}
    \end{equation}
    Finally, consider a compact set $K \subset \subset \Omega_{0,1}$ such that $\cup_{n=1}^{\infty} B(y(x_n),R/2) \subset K$. Since we are assuming $u_{1,n} (x_n) < k$ for all $n$ and since $u_{1,n}(x)$ tends uniformly to $\infty$ in $K$, we have that
    \begin{equation}
        \label{convergence_c_n}
        c_n \to \infty,
    \end{equation}as $n \to \infty$ (or equivalently, when $\lambda_n$ goes to $\lambda_\infty$). Moreover, if we set the normalised direction
    $$
    \boldsymbol{\nu}_n :=\frac{y(x_n)-x_n}{R},
    $$then, using \eqref{inequality_proof_blow_up}, the partial derivative in that direction yields
    \begin{equation*}
        \begin{aligned}
            \frac{\partial u_{1,n}}{\partial \boldsymbol{\nu}_n}(x_n) &= \lim_{t \to 0} \frac{u_{1,n}(x_n+t \boldsymbol{\nu}_n ) - u_{1,n}(x_n)}{t} \geq c_n \lim_{t \to 0} \frac{w_n(x_n+t \boldsymbol{\nu}_n )}{t} \\ &= c_n \lim_{t \to 0}\frac{e^{-\delta \left| x_n + t \boldsymbol{\nu}_n - y(x_n) \right|^2} - e^{-\delta R^2}}{t} = c_n \lim_{t \to 0}\frac{e^{-\delta \left| t \boldsymbol{\nu}_n - \boldsymbol{\nu}_n R \right|^2}- e^{-\delta R^2}}{t} \\ &\geq
            c_n \lim_{t \to 0}\frac{e^{-\delta( t - R)^2} - e^{-\delta R^2}}{t} = c_n \lim_{t \to 0} (-2 \delta) e^{-\delta (t-R)^2}(t-R) = c_n 2\delta R e^{-\delta R^2}.
        \end{aligned}
    \end{equation*}As a consequence of \eqref{convergence_c_n},
    \begin{equation}
    \label{lim_inf_theorem_blow_up}
        \lim_{n \to \infty}  \frac{\partial u_{1,n}}{\partial \boldsymbol{\nu}_n}(x_n) = \infty.
    \end{equation}On the other hand, if we prove that
    \begin{equation}
         \frac{\partial u_{1,n}}{\partial \boldsymbol{\nu}_n}(x_n)  \leq K<\infty,
\label{contradiction_theorem_blow_up}
    \end{equation}
    a contradiction to \eqref{lim_inf_theorem_blow_up} arises, and we conclude that $u_{1,n} \to \infty$ in $\overline{\Omega}_{0,1}$ as we desire.
    Thus, let $v_n$ be the unique solution to
        \begin{equation}
    \label{problem_aux_proof_blow_up}
        \begin{cases}
            \begin{aligned}
                & -\Delta u = \lambda_n m_1(x) u - a_1(x) u^p \quad && \text{in} \quad \Omega_1 \setminus \overline{\Omega}_{0,1}, \\
                & u = u_{1,n}(x_n) \quad && \text{on} \quad \partial \Omega_{0,1}, \\
                & \frac{\partial u }{\partial \boldsymbol{\nu_1}} + \gamma_1 u = 0 \quad &&\text{on} \quad \Gamma,
            \end{aligned}
        \end{cases}
    \end{equation} see Lemma~\ref{lemma:auxiliary.complete}. We have that $u_{1,n}$ is a supersolution of problem \eqref{problem_aux_proof_blow_up} and by comparison
    \begin{equation}
    \label{bound_v_proof_blow_up}
        v_n(x) \leq u_{1,n}(x), \quad \text{for any} \quad x \in \overline{\Omega}_1 \setminus \Omega_{0,1}.
    \end{equation}Consider now $v_\infty$ the unique positive solution of
    \begin{equation*}
        \begin{cases}
            \begin{aligned}
                & -\Delta u = \lambda_\infty m_1(x) u - a_1(x) u^p \quad && \text{in} \quad \Omega_1 \setminus \overline{\Omega}_{0,1}, \\
                & u = k \quad && \text{on} \quad \partial \Omega_{0,1}, \\
                & \frac{\partial u }{\partial \boldsymbol{\nu_1}} + \gamma_1 u = 0 \quad &&\text{on} \quad \Gamma.
            \end{aligned}
        \end{cases}
    \end{equation*} Since $u_{1,n}(x_n)\leq k$, recall~\eqref{sequence_contradiction_theorem}, $v_\infty$ is a supersolution to \eqref{problem_aux_proof_blow_up} and $v_n \leq v_\infty$. Hence  $\| v_n \|_{L^\infty (\overline{\Omega}_1 \setminus \Omega_{0,1} )}$ is uniformly bounded, independently
     of $n$. Thanks to the $L^p$ estimates and the Sobolev embedding theorem, we have that $\{ v_n\}$ is a bounded sequence in $C^{1,\eta} (\overline{\Omega}_1 \setminus \Omega_{0,1} )$, and hence, $\| \nabla v_n \|_{L^{\infty} (\overline{\Omega}_1 \setminus \Omega_{0,1} )} \leq C$, for a positive constant $C$, see for instance~\cite{GT}. Lastly, due to \eqref{bound_v_proof_blow_up} and since $v_n(x_n) = u_{1,n}(x_n)$ we conclude
    $$
    \frac{\partial u_{1,n}}{\partial \boldsymbol{\nu}_n}(x_n)  \leq \frac{\partial v_{n}}{\partial \boldsymbol{\nu}_n}(x_n)  \leq C,
    $$which actually proves \eqref{contradiction_theorem_blow_up} and finishes the proof.
\end{proof}

 \begin{theorem}
     \label{theorem:blow_up_u2 convergence}
         Let $\lambda_\infty:=\lambda_{m_1}<\lambda_{m_2}$. For any fixed $\lambda\in (\lambda_*,\lambda_\infty)$, let ${\bf u_\lambda}$ be the unique positive solution of \eqref{eq:original_problem}--\eqref{eq:conditions}. Then
      $$u_{2,\lambda}\leq M\ \text{in  $\overline{\Omega}_2$},
    $$
   where $M$ is a positive constant.
      \end{theorem}

\begin{proof} First, observe that due to the a priori bounds, since $\Omega_{\epsilon_1,i}\subset \Omega_{\epsilon_2,i}$ for $\epsilon_1< \epsilon_2$ we find that Lemma~\ref{lemma.apriori} provides us with a priori bounds of ${\bf u_\lambda}$ in $\Omega\setminus \overline\Omega_{0,1}$. Using now the bound for $u_{1,\lambda}$ on $\Gamma\cup \partial \Omega$ we find that $u_{2,\lambda}$ is a subsolution of~\eqref{eq:problem_aux} in $\Omega_2$, just choosing a sufficiently large constant $C$ on the interface region. Subsequently, we might construct a bounded supersolution $\phi_K$ to~\eqref{eq:problem_aux}, just following the construction made in the  proof of Theorem~\ref{eq:lambda_condition_nondegenerate} and defining $\Phi$  as in~\eqref{def.psi.soper}. Then, by comparison $u_{2,\lambda}\leq \phi_K$ in $\Omega_2$. We skip the details by repetitives.
\end{proof}

Finally, we analyse what happens outside the set $\Omega_{0,1}$. Consider the problem
\begin{equation}
\label{prob_aux_final_blow_up}
    \begin{cases}
        \begin{aligned}
            &-\Delta u_1 = \lambda m_1(x)  u_1 -  a_1(x)u_1^p \quad &&\text{in} \quad \Omega_1 \setminus \overline{\Omega}_{0,1}, \\
            &-\Delta u_2 = \lambda m_2(x)  u_2 -  a_2(x)u_2^p \quad &&\text{in} \quad \Omega_2, \\
            & \mathcal{I}({\bf u}) = 0 \quad &&\text{on} \quad \Gamma\cup \partial \Omega,
        \end{aligned}
    \end{cases}
\end{equation}
with the additional condition on the boundary of $\partial \Omega_{0,1}$  given by
\begin{equation}
\label{cond_final_blow_up}
u_1 = \infty \quad \text{on} \quad  \partial \Omega_{0,1}.
\end{equation}
As is common in the literature,~\eqref{cond_final_blow_up} means that $u_1(x) \to \infty$ as dist$(x, \partial \Omega_{0,1}) \to 0,$ see~\cite{DuHuang1999} for an uncoupled problem or~\cite{AlvarezBrandleMolinaSuarez2025} for an interface system.
Then, we next show the existence of solutions to \eqref{prob_aux_final_blow_up}--\eqref{cond_final_blow_up}. To this end, let us first state the existence of solutions for an auxiliary problem.
\begin{lemma}
    For any $\lambda \in(\lambda_*,\lambda_{m_2})$, Problem \eqref{prob_aux_final_blow_up}
    with the condition
    \begin{equation}
\label{cond:K_convergence_refugio1}
u_1=\phi \quad \text{in} \quad \partial \Omega_{0,1}, \quad \hbox{with}\quad \phi\in C^{2,\eta}(\partial\Omega_{0,1}).
\end{equation}
possesses a unique positive solution denoted by $\tilde {\bf u}$.
\end{lemma}
\begin{proof}
    The proof follows the arguments described in \cite[Lemma 2.3]{DuHuang1999} and \cite{MV} and is based on the construction of an appropriate sub and supersolution.

It is clear that $\underline{\bf v}\equiv 0$ is a subsolution of \eqref{prob_aux_final_blow_up}--\eqref{cond:K_convergence_refugio1}. To construct a supersolution we define an auxiliary problem that avoids the condition~\eqref{cond:K_convergence_refugio1} and for which we know how to characterize the positive solutions. To do so, define
$$
\overline{\bf a}=\left\{ \begin{aligned}
    \overline{a_1},\quad x\in\Omega_1,\\
    a_2,\quad x\in\Omega_2,
\end{aligned}\right.
$$
where $\overline{a_1}$ is non-negative in $\Omega_1$ and such that $ \overline{a_1}\leq a_1$ in $\Omega_1\setminus\Omega_{0,1}$. Moreover, we define $\overline{a_1}$ so that the set $\mathcal{U}=\{x\in\Omega_1 : \overline{a_1}=0\}$ is small enough, ensuring that the first eigenvalue $\lambda_m^\mathcal{U}[-\Delta ;\mathcal{D}] >\lambda_{m_2}$.
Consider
$$
\left\{\begin{array}
  {ll} -\Delta \mathbf{u}=\lambda {\mathbf{mu}}-\overline{\bf a}{\bf u}^p,\quad &x\in\Omega,\\
  \mathcal{I}(\mathbf{u})=0,&x\in\Gamma\cup \partial \Omega.
\end{array}\right.
$$
Due to the definition of $\overline {\bf a}$, we conclude, see Theorem~\ref{theorem:solution_degenerate}, that there exists a unique positive solution ${\bf v}$ to this latter system as long as $\lambda\in(\lambda_*,\lambda_{m_2})$.
  Under these conditions it is clear that $\overline{\bf v}=M {\bf v}$, with $M>1$ big enough so that $Mv_1>\phi$ on $\partial\Omega_{0,1}$, is a supersolution of \eqref{prob_aux_final_blow_up}--\eqref{cond:K_convergence_refugio1}.

   Therefore, having a sub and a supersolution  and also thanks to Lemma~\ref{aux_lemma_minimal} we get that there is a unique solution $\tilde{\bf u}$ for~\eqref{prob_aux_final_blow_up}--\eqref{cond:K_convergence_refugio1}.
  \end{proof}

\begin{remark}
    An analogous argument taking $\overline{a_2}\leq a_2$ allows to construct a global solution in $\Omega\setminus(\Omega_{0,1}\cup\Omega_{0,2})$ for all $\lambda>\lambda_*$. Moreover, following~\cite{DuHuang1999} one can extend the result for all  $\lambda\in\mathbb{R}$, by  fixing $\bar\lambda =\max\{\lambda, \lambda_*\}$ and analysing the problem given by the equations
      $$
      -\Delta {\bf u} = \bar \lambda {\bf m}{\bf u} -  \overline{\bf a}{\bf u}.
      $$
  \end{remark}

Now, we can prove the existence of a minimal solution for problem \eqref{prob_aux_final_blow_up}--\eqref{cond_final_blow_up}.

 \begin{lemma}
\label{lemma:minimal}
    For any $\lambda \in(\lambda_*,\lambda_{m_2})$, Problem \eqref{prob_aux_final_blow_up}--\eqref{cond_final_blow_up} has a minimal solution $\underline{\bf u}$, in the sense that any solution $\bf u$ of \eqref{prob_aux_final_blow_up}--\eqref{cond_final_blow_up} satisfies $\underline{\bf u} \leq \bf u$.
\end{lemma}

\begin{proof}
Take the solution $\tilde{\bf u}_k$ of problem \eqref{prob_aux_final_blow_up}--\eqref{cond:K_convergence_refugio1}, such that $\phi(x)=k>0$, i.e. $\tilde{u}_{1,k}=k$ on $\partial \Omega_{0,1}$ Then, by comparison, see Lemma~\ref{aux_lemma_minimal} and~\cite{WangSu}, we have that the sequence $\{\tilde{\bf u}_k\}$ is monotonically increasing. The goal is to show that the solution $\tilde{\bf u}_k$ is bounded above by some function ${\bf v}$ which is uniformly bounded on all compact subsets of $\overline{\Omega} \setminus {\Omega}_{0,1}$. Under these assumptions, we define $\underline{\bf u} = \lim_{k \to \infty} \tilde{\bf u}_k$. Using simple regularity and compactness argument we will have that $\underline{\bf u}$ is indeed  a solution of \eqref{prob_aux_final_blow_up}--\eqref{cond_final_blow_up}.

Hence, to define the function ${\bf v}$ we follow the construction done in the proof of Theorem~\ref{theorem:solution_degenerate}, see~\eqref{def.psi.soper}, considering eigenfunctions ${\bf \Psi}$ for domains surrounding $\Omega_{0,i}$. We omit the details of the construction.

Next, consider a large constant $M>0$ and define ${\bf v}=M{\bf \Psi}$ so that
$$\lim_{d(x,\Omega_{0,1})\to 0} (\tilde{u}_{1,k}(x) - v_1(x) ) = - \infty <0,$$
Also, it is clear by construction that
$$
 -\Delta {\bf v} - \lambda {\bf mv}+{\bf a v}^p \geq 0, \hbox{\quad for }\quad x\in\Omega
$$
and
$$\mathcal{I}({\bf v}) = \mathcal{I}(M{\bf \Psi}) = 0,\quad \ \hbox{for }\quad x \in\Gamma \cup \partial \Omega.$$
Hence, we can apply Lemma \ref{aux_lemma_minimal} and conclude that ${\bf \tilde{u}_{k}} \leq {\bf v}$ in $\overline{\Omega} \setminus {\Omega}_{0,1}$.

Finally, let ${\bf u}$ be a solution of~\eqref{prob_aux_final_blow_up}--\eqref{cond_final_blow_up}. By a comparison argument, we have that ${\bf u} \geq \tilde{\bf u}_k$ for all $k>0$.
Therefore, passing to the limit,
$$\underline{\bf u} \equiv \lim_{k \to \infty} \tilde{\bf u}_k \leq {\bf u}$$ and $\underline{\bf u}$ is a minimal positive solution of \eqref{prob_aux_final_blow_up}--\eqref{cond_final_blow_up}.
\end{proof}

\begin{remark}
    Problem \eqref{prob_aux_final_blow_up}--\eqref{cond_final_blow_up} also admits a maximal positive solution. Indeed, to prove it one can define problem \eqref{prob_aux_final_blow_up}--\eqref{cond_final_blow_up} in the domain $\Omega_{1}\setminus \overline \Omega_{1,n}$ such that
    $$\Omega_{1,n}:=\{x\in\Omega_1\,;\, d(x,\Omega_{0,1})<1/n\},$$
    and, constructing a decreasing sequence of solutions $\{\hat{\bf u}_n\}$, i.e. $\hat{\bf u}_n\geq \hat{\bf u}_{n+1} \geq {\bf u}$ on $\Omega_1\setminus \overline\Omega_{1,n}$ we arrive at $\overline{\bf u}\geq \lim_{k\to \infty} {\bf u}_{k} ={\bf u}$; see for further details \cite{DuHuang1999}.
\end{remark}

}

\begin{theorem}
\label{theorem:lambda_inf_m1_eq_m2}
    Let $\lambda_\infty=\lambda_{m_1}<\lambda_{m_2}$. For any fixed $\lambda\in (\lambda_*,\lambda_\infty)$, let ${\bf u_\lambda}$ be the unique positive solution of \eqref{eq:original_problem}--\eqref{eq:conditions}. Then
       $${\bf u}_\lambda \to {\bf u}_{\lambda_\infty}$$ uniformly on compact subsets of $\overline{\Omega}\setminus \overline{\Omega}_{0,1}$, as $\lambda \to \lambda_\infty$, where ${\bf u}_{\lambda_\infty}$ is the minimal positive solution of \eqref{prob_aux_final_blow_up}--\eqref{cond_final_blow_up} with $\lambda=\lambda_\infty$.
\end{theorem}

\begin{remark}
 Since $\lambda_\infty =\lambda_{m_1}<\lambda_{m_2}$ we would like to note that, when passing to the limit as $\lambda$ goes to $\lambda_\infty$, regarding the equation in $\omega_2$ we find that
    $$
    -\Delta u_2 = \lambda_{\infty} m_2(x) u_2 - a_2(x)u_2^p \quad \text{in} \quad \Omega_2,
    $$ which has a positive (bounded) solution $u_{2,\infty}$ for $\lambda< \lambda_{m_2}$.
\end{remark}

\begin{proof}
    Consider an increasing sequence $\{\lambda_n\}$ which converges to $\lambda_\infty$ as $n \to \infty$, and the corresponding unique positive solutions ${\bf u}_{\lambda_n}$ of \eqref{eq:original_problem}--\eqref{eq:conditions}, which we denote by ${\bf u}_n$.

    It is clear that, by monotonicity of the solution with respect to $\lambda$, Theorem~\ref{theorem:bifurcation},  we have that ${\bf u}_n\leq {\bf u}_{n+1}$. On the other hand, due to  Theorem~\ref{theorem:blow_up_u2 convergence} we have that $u_{2,n}$ is uniformly bounded on $\overline{\Omega}_2$.

    To see that $u_{1,n}$ is also
         uniformly bounded on every compact subset of $\overline{\Omega}_1\setminus \overline{\Omega}_{0,1}$
          we consider, for a sufficiently small $\epsilon>0$, the $\epsilon$-neighbourhood $\Omega_{\epsilon,1}$ given in~\eqref{def.omega.eps}. Let $a_\epsilon$ be a positive constant, such that $a_1(x) \geq a_\epsilon>0$ in $\overline{\Omega}_{\epsilon,1 }$ and let us consider, for a fixed $n_0$, the problem
   \begin{equation}
    \label{prob_aux_last_proof}
        \begin{cases}
            \begin{aligned}
            &-\Delta u_1 = \lambda_\infty m_1(x) u_1 - a_\epsilon u_1^p \quad &&\text{in} \quad \Omega_{\epsilon,1}, \\
            &\frac{\partial u_1 }{\partial \boldsymbol{\nu_1}} + \gamma_1 u_1  =\max u_{2,n_0}(x) \quad &&\text{on} \quad \Gamma\\
            &u_1 = u_{1,n_0}&& \text{on} \quad \partial \Omega_{\epsilon,1}.  \end{aligned}
        \end{cases}
    \end{equation}
 By construction, ${u}_{1,n_0}$ is a subsolution for problem \eqref{prob_aux_last_proof}. Moreover, if we replace the above condition on $\partial\Omega_{\epsilon,1}$, by $u_1=\infty$, then the minimal large solution to that problem provides us with a supersolution, $v_1$, of \eqref{prob_aux_last_proof}.
   Then,  by the comparison principle we have $$ v_1 \geq u_{1,n_0} \quad \text{on} \quad \Omega_{\epsilon,1}.$$
    Furthermore, since $v_1$ is bounded in $\overline{\Omega}_{2\epsilon,1 }$, see \cite{DuHuang1999} and the above resutls, there exists a positive constant such that $u_{i,n} \leq C$ in $\overline{\Omega}_{2\epsilon,i}$, for all $n \geq 1$. Since $\epsilon >0$ is arbitrary, this implies that $u_{1,n_0}$ is uniformly bounded on compact subsets of $\overline{\Omega}_1\setminus \overline{\Omega}_{0, 1}$. Therefore, we can pass to the limit as $n\to\infty$ and get that ${\bf u}_n$ converges to a limit function ${\bf u}_{\lambda_\infty}$ in $\overline{\Omega}\setminus \overline{\Omega}_{0,1} $. Moreover, by regularity we can pass to the limit in problem \eqref{eq:original_problem} and~\eqref{eq:conditions} and conclude that ${\bf u}_{\lambda_\infty}$ actually verifies~\eqref{prob_aux_final_blow_up} with $\lambda=\lambda_\infty$. To get the condition \eqref{cond_final_blow_up} we suppose by contradiction that
 $$
 \lim_{\text{dist} (x, \partial \Omega_{0,1})} u_{1,\lambda_\infty}=\infty, \quad \text{uniformly for} \quad x \in \overline{\Omega}_1 \setminus \Omega_{0,1},
 $$does not hold. That is,  there exists a sequence $x_n  \in \overline{\Omega}_1 \setminus \Omega_{0,1}$ converging to $x_\infty \in \partial \Omega_{0,1}$, such that $u_{1,\lambda_\infty}(x_n) \leq C$ for any $n \geq 1$ and some positive constant $C$. Hence, we have $u_{1,k}(x_n) \leq C$ for any $n \geq 1$ and $k \geq 1$, since ${\bf u}_n$ converges from below to $ {\bf u}_{ \lambda_\infty}$ as $n \to \infty$.

    On the other hand, thanks to Theorem \ref{theorem:blow_up_degenerate_2}, we also know that $u_{1,k}(x_\infty) \to \infty$ as $k \to \infty$, uniformly for all $n\geq 1$. Thus, there exists $k_0$ sufficiently large $u_{1,k_0}(x_\infty) \geq 3C$ for all $n \geq 1$. Since $u_{1,k_0}$ is uniformly continuous we have that $\left| u_{i,k_0}(x_n) - u_{i,k_0}(x_\infty) \right| \leq C$. Equivalently,
 $$
 u_{1,k_0}(x_n) \geq u_{1,k_0}(x_\infty) - C \geq 2C,
 $$
which contradicts the assumption $u_{1,k}(x_n) \leq C$ for all $n \geq 1$ and $k \geq 1$.

Finally, to show $ {\bf u}_{\lambda_\infty}$  is actually minimal postive solution of  \eqref{prob_aux_final_blow_up}--\eqref{cond_final_blow_up} with $\lambda=\lambda_\infty$ consider any other solution ${\bf \hat{u}}_{\lambda_\infty}$. By comparison, ${\bf u}_n < {\bf \hat{u}}_{\lambda_\infty}$ in $\overline{\Omega} \setminus  \overline{\Omega}_{0,1}$. Therefore, letting $n \to \infty$ we conclude that ${\bf u}_{\lambda_\infty} <  {\bf \hat{u}}_{\lambda_\infty}$.
\end{proof}

\end{document}